\documentclass[11pt]{amsart}
\usepackage[utf8]{inputenc}
\usepackage[margin=1.15in]{geometry}
\usepackage{amsthm, amsmath, amsfonts, amssymb, array, bm, bbm, caption, commath, enumitem, graphics, graphicx, hyperref, ifthen, latexsym, mathtools, pgf, pgfplots, setspace, subcaption, tikz, tkz-euclide, wrapfig, xcolor}
\pgfplotsset{compat=1.18}
\usepackage[textsize=small]{todonotes}
\usepackage[alphabetic]{amsrefs}
\title{Thin surface subgroups of non-uniform arithmetic lattices in \texorpdfstring{$\operatorname{SO}^+(\lowercase{n},1)$}{SO+(n,1)}}
\author{Sara Edelman-Mu\~noz}
\address[S. Edelman-Mu\~noz]{Durham Veterans Affairs HCS, HSR (152), 411 W. Chapel Hill St, Durham 27701}
\email{\href{maiilto:sedelmanmunoz@gmail.com}{sedelmanmunoz@gmail.com}}
\author{Michael Zshornack}
\address[M. Zshornack]{Northwestern University, 2033 Sheridan Road, Evanston, IL 60208, United States}
\email{\href{mailto:zshornack@northwestern.edu}{zshornack@northwestern.edu}}

\date{\today}

\newtheorem{theorem}{Theorem}
\newtheorem{lemma}[theorem]{Lemma}
\newtheorem{corollary}[theorem]{Corollary}

\theoremstyle{definition}
\newtheorem{definition}{Definition}

\newcommand{\SOqZ}{\textnormal{SO}^+(q,\mathbb{Z})}
\newcommand{\SOfZ}{\textnormal{SO}^+(f,\mathbb{Z})}
\newcommand{\legendre}[2][p]{\ensuremath{\left( \frac{#2}{#1} \right) }}
\newcommand{\fold}{\ensuremath{FM}}
\newcommand{\h}{\mathbb{H}}

\newcolumntype{C}{>{{}}c<{{}}}

\usetikzlibrary{calc}
\pgfkeys{/pgf/declare function={arccosh(\x) = ln(\x + sqrt(\x^2-1));}}

\tikzset{->-/.style={decoration={\
            markings,
            mark=at position #1 with {\arrow{>}}},postaction={decorate}}}

\tikzset{arrowMe/.style={postaction=decorate,
      decoration={markings, mark=at position 0.8 with {\arrow[thick]{#1}}
      }}}

\newcommand{\arcThroughThreePoints}[4][]{
            \tkzCircumCenter(#2,#3,#4)\tkzGetPoint{c}
            \tkzDrawArc[#1](c,#2)(#4)
      }

\begin{document}

\begin{abstract}
    We show that the fundamental groups of all non-compact, arithmetic, hyperbolic, $n$-manifolds for $n \geq 4$ contain thin surface subgroups. As a consequence of the proof of this theorem we also show that the fundamental groups of the doubles of cusped, arithmetic, hyperbolic $n$-manifolds embed as GFERF subgroups of $\operatorname{SO}^+(n+1,1)$.
\end{abstract}

\maketitle

\section{Introduction}
\label{ch:Intro}

Let $n\geq 3$ and $\Gamma\leqslant\operatorname{SO}^+(n,1)$, be a lattice. A \emph{surface subgroup} of $\Gamma$ is a subgroup isomorphic to the fundamental group of a closed surface of genus $g\geq 2$. Understanding the existence and geometric features of surface subgroups of $\Gamma$ has been of interest to topologists and geometers for many years. This is, in part, due to the relationship between such subgroups and incompressible surfaces in the ambient hyperbolic manifold. When such a subgroup exists, it is always infinite-index in $\Gamma$ and can often be made Zariski-dense in $\operatorname{SO}^+(n,1)$. When $\Gamma$ is arithmetic, such a subgroup is said to be \emph{thin} in the sense of \cite{Sarnak}.

Thin subgroups of arithmetic groups have attracted independent interest in recent years within number theory and combinatorics as these subgroups share a number of useful properties with the arithmetic group they live in. For instance, such subgroups satisfy forms of \emph{superstrong approximation}: they surject onto almost all congruence quotients of the lattice and their induced Cayley graphs form expander families (see \cite{SGV}). Despite the rich emerging theory behind the properties of thin groups, it is not particularly well understood how such examples arise. Though thin \emph{free} subgroups are abundant in many contexts, (c.f., \cite{fuchs}), methods from hyperbolic geometry have proven especially useful in constructing examples of varying isomorphism type. In \cite{ballaslong}, the authors draw attention to the general lack of non-free examples and raise the question of which isomorphism types of groups can arise as a thin subgroup of an arithmetic group.

Due to the interest in constructing thin groups of non-free isomorphism type and the relationship between surface subgroups of lattices and incompressible surfaces in locally symmetric spaces, we focus on constructing Zariski-dense surface subgroups of fundamental groups of hyperbolic manifolds. In dimension $3$, such examples are known: all lattices in $\operatorname{PSL}(2,\mathbb{C})$ contain thin surface subgroups. For uniform lattices, this follows from \cite{KM} while for non-uniform lattices, it follows from \cite{CLR}. In higher dimensions, the existence of thin surface subgroups is known for uniform lattices in $\operatorname{SO}^+(n,1)$ for odd $n$ by \cite{Hamenstadt} and for even $n$ by recent work in \cite{kahnrao}. What remains open is the case of non-uniform lattices when $n>3$. Presently, there is only example of a commensurability class of non-uniform lattices in each dimension known to contain a thin surface subgroup \cite{Douba}. Notably, these lattices are all arithmetic. In this paper, we extend these results by proving the following theorem.

\begin{theorem} 
\label{11}
Let $\Gamma$ be a non-uniform, arithmetic lattice in $\textnormal{SO}^{+}(n,1)$, with $n \geq 4$. Then $\Gamma$ has a thin surface subgroup.
\end{theorem}

The proof of Theorem~\ref{11} is an induction on the dimension $n$, beginning with $n=4$. We first find a subgroup of $\Gamma\leqslant\operatorname{SO}^+(4,1)$ isomorphic to the fundamental group of a cusped, arithmetic hyperbolic $3$-orbifold. Because this $3$-orbifold virtually fibers, it has a finite cover, $M$, with an embedded punctured subsurface $\Sigma$. We define the \emph{folded double} of $M$, $FM$, a graph of spaces whose fundamental group contains the fundamental group of the double of $\Sigma$, $D\Sigma$. We then describe how $FM$ maps into the $4$-orbifold $\mathbb{H}^4/\Gamma$ and show that the induced map $\pi_1(FM)\to\Gamma$ is faithful and its image Zariski-dense. Since $\pi_1(D\Sigma)$ is embedded in $\pi_1(FM)$, the result is a thin surface subgroup of $\Gamma$. Finally, we generalize to higher dimensions.

In the interest of the questions in \cite{ballaslong} on constructing \emph{any} non-free examples of thin subgroups of $\Gamma$ as above, we note that the proof additionally produces thin subgroups of $\Gamma$ isomorphic to $\pi_1(FM)$.

As a further application of the methods of Theorem~\ref{11}, we also prove a subgroup separability result about the groups considered. Recall that a discrete group of isometries of $\mathbb{H}^n$ is said to be \emph{GFERF} if it is subgroup separable on its geometrically finite subgroups. We use the folded double construction of Theorem~\ref{11} and results of \cite{BHW} to prove the following theorem.

\begin{theorem} 
\label{12}
Let $M$ be a cusped, arithmetic, hyperbolic $n$-orbifold and let $DM$ be the double of $M$ over its cusps. Then $\pi_1(DM)$ embeds discretely into $\operatorname{SO}^+(n+1,1)$ and its image is GFERF.
\end{theorem}

Note that the property of a group being GFERF is weaker than being \emph{LERF}, i.e., being subgroup separable on all finitely generated subgroups. Nonetheless, there has been interest in both properties because of their use in lifting immersed incompressible submanifolds to embedded ones in a finite cover. More intrinsic versions of this theorem, not reliant on embeddings into $\operatorname{SO}^+(n,1)$ are also discussed at the end of \S\ref{ch:GFERF}.

In the course of the proof of Theorem~\ref{12}, we also show that LERF is too restrictive a property for these groups. Though the fundamental groups of hyperbolic $3$-manifolds are LERF by \cites{Agol2,Wise}, the fundamental groups of higher dimensional arithmetic hyperbolic manifolds are not by \cites{Sun,BSS}. Since, also by \cite{Sun}, the fundamental groups of doubles of non-compact, arithmetic, hyperbolic $3$-manifolds over their cusps are not LERF, showing these groups are GFERF is significant. The starting point for Theorem~\ref{12} is the work of Bergeron--Haglund--Wise in \cite{BHW} showing that fundamental groups of arithmetic hyperbolic manifolds of simplest type are GFERF\@. From there, we generalize the methods of \cite{LR}, where it was shown that the double of the figure-$8$ knot complement over its cusps is GFERF, extending this construction to the doubles of all non-compact arithmetic hyperbolic $n$-manifolds.

\subsubsection*{Acknowledgements} This work was part of the first author's Ph.D.\ thesis. She thanks her advisor Alan Reid for many years of guidance and support. The authors also thank Daniel Groves and Ryan Spitler for ideas that contributed directly to this work and Sami Douba, Chris Leininger and Darren Long for their comments and suggestions. The second author was supported by NSF grant DMS-2136217.

\section{Non-uniform Arithmetic Subgroups of \texorpdfstring{$\operatorname{SO}^+(n,1)$}{SO+(n,1)}}
\label{groups.and.qf}
We first discuss arithmetic groups defined by quadratic forms to motivate how to find subgroups isomorphic to fundamental groups of cusped $3$-manifolds. Recall that if $X$ is a non-uniform hyperbolic $n$-orbifold, its fundamental group, $\Gamma<\operatorname{SO}^+(n,1)$, acts on $\h^n$ and admits a non-compact fundamental domain of finite-volume. If, in addition, $\Gamma$ is arithmetic, then $\Gamma$ has a conjugate commensurable to a group of the form $\operatorname{SO}^+(q,\mathbb{Z})$ for some integral diagonal quadratic form, $q$, of signature $(n,1)$ (this fact follows from the Tits classification, but see, for instance, \cite{Witte}*{Proposition 6.4.2}). Recall that $\operatorname{SO}^+(q,\mathbb{Z}):=\{g\,|\, g\in\operatorname{SL}(n,\mathbb{Z}),\ g^T A_q g=A_q\}$ where $A_q$ is the diagonal matrix whose entries are the coefficients of $q$. Since, to prove Theorem~\ref{11}, it suffices to show that a group commensurable to $\Gamma$ up to conjugacy contains a thin surface subgroup, we will instead work with the more explicit groups $\operatorname{SO}^+(q,\mathbb{Z})$ for suitable $q$.

Specializing to the case where $n=4$, the starting point of our induction, we leverage the fact that we have explicit representatives for quadratic forms in each commensurability class of non-uniform, arithmetic lattices. First, we fix some notation.

\begin{definition}
    For integers $a_1,\ldots, a_n$, we let $q=\langle a_1,\ldots,a_n\rangle$ denote the diagonal quadratic form defined by $q(x_1,\ldots,x_n)=a_1 x_1^2+\cdots+a_n x_n^2$.
\end{definition}

From \cite{Montesinos}, when $n=4$, the commensurability class of $\operatorname{SO}^+(q,\mathbb{Z})$ is uniquely determined by the projective equivalence class of $q$. Furthermore, from \cite{Montesinos}, we can take $q$ to be
\[
q=\begin{cases}
    \langle -1, 1, 1, aS, a\rangle & S\equiv 1\pmod{4}\\
    \langle 1,1,1,aS,-a\rangle & S\equiv -1\pmod{4}
\end{cases}
\]
where $S$ is the product of $n$ distinct odd primes and $a$ is an odd prime such that $a\nmid S$. Additionally, we may require that $\legendre[p_i]{a} = -1$ for all primes $p_i$ dividing $S$, that $a\equiv (-1)^n\pmod{4}$ when $S\equiv 1\pmod{4}$ and that $a\equiv (-1)^{n+1}\pmod{4}$ when $S\equiv -1\pmod{4}$.

\begin{lemma} \label{lemma:groups.and.qf}
    Let $q$ be as above. The group $\SOqZ$ contains some subgroup that is isomorphic to the fundamental group of a non-compact arithmetic hyperbolic 3-manifold or orbifold defined via a quadratic form. 
\end{lemma}

\begin{proof}

We will show that for each projective equivalence class there is some $q$ that contains some diagonal subform $f$ that is isotropic and has signature $(3,1)$. Then the representation of $\SOfZ$ into $\SOqZ$ defined by $g \mapsto g \oplus [1]$ forms our subgroup. We refer to this representation as the \emph{upper left corner} representation. 

When $q =  \langle -1,1,1,aS,a \rangle$, it has the subform $f =  \langle -1,1,1,a \rangle$. One can see that $f$ is isotropic because the quadratic equation $-x_0^2+x_1^2+x_2^2+ax_3^2=0$ has the non-trivial solution $(1,1,0,0)$. 

On the other hand when $q =  \langle 1,1,1,aS,-a \rangle$, it has the subform $f =  \langle 1,1,1,-a \rangle$. By Legendre's three square theorem the equation $x_0^2+x_1^2+x_2^2-ax_3^2=0$ has a solution if and only if $a \not\equiv 7 \mod 8$. In this case $f$ is isotropic. 

The only case that remains is when $S \equiv -1 \mod 4$ and $a \equiv 7 \mod 8$. We will deal with this case by demonstrating that there is some form $q' = \langle 1,1,1,a'S,-a' \rangle$ that is projectively equivalent to $q$ such that $a' \equiv 3 \mod 8$, allowing us to use $\langle 1,1,1,-a'\rangle$ as an isotropic subform. 

By the Hasse-Minkowski theorem, two quadratic forms are equivalent over $\mathbb{Q}$ if their signature, discriminant, and Hasse-Minkowski invariants $c_p(\cdot)$ are the same \cite{Serre}. When $S \equiv -1 \mod 4$ and $a \equiv 3 \mod 4$, the Hasse-Minkowski invariants of $q$ can be shown\footnote{See Appendix~\ref{append:hasse} for Hasse-Minkowski Invariant computations.} to be:

\[c_p(q)= \begin{cases} 
      1 & p = 2 \\
      1 & p = a \\
      -1 & p | S \\
      1 & \textrm{otherwise.} \\
   \end{cases}
\] 
Notice that under the above assumptions, the Hasse-Minkowski invariants rely only on $S$. Furthermore, the discriminant $\Delta(q) = S$. Now suppose that $q =  \langle 1,1,1,aS,-a \rangle$ with $a \equiv 7 \mod 8$. Because the discriminant and Hasse-Minkowski invariants rely only on $S$ and the congruence class of $a \mod 4$, we simply need to find some $a' \equiv 3 \mod 8$ such that $a'$ is an odd prime that does not divide $S$ and $\legendre[p_i]{a'} = -1$ for any prime $p_i$ that divides $S$. Then $q' = \langle 1,1,1,a'S,-a' \rangle$ is equivalent to $q$. 

We can employ classical results to show that a desired $a'$ exists. By Sunzi's remainder theorem\footnote{Formerly known as the Chinese Remainder Theorem.} \cite{Gauss}, there is some $n<8S$ such that $n \equiv a \mod S$ and $n \equiv 3 \mod 8$. Furthermore, because $a$ is coprime to $S$, and 3 is coprime to 8, $n$ is coprime to $8S$. Then by Dirichlet's theorem \cite{Serre} there are infinitely many primes $a'$ such that $a' \equiv n \mod 8S$. These primes meet the necessary conditions to ensure that $q'$ is equivalent to $q$. 
\end{proof} 

{Note that the subgroup of $\operatorname{SO}^+(q,\mathbb{Z})$ constructed in this lemma is never Zariski-dense as its Zariski-closure is always isomorphic to some copy of $\operatorname{SO}^+(3,1)$ inside $\operatorname{SO}^+(4,1)$. In the next section, we show how one can modify this subgroup to construct a thin surface subgroup inside $\operatorname{SO}^+(q,\mathbb{Z})$.} By commensurability this will give us a thin surface subgroup in $\Gamma$. 

\section{Doubles of Fibered 3-Manifolds in 4-Manifolds and Orbifolds} \label{double}
The representation we constructed in Section~\ref{groups.and.qf} induces an immersion of a cusped, arithmetic, hyperbolic 3-orbifold $M^3$ into our 4-orbifold $X$. As previously mentioned these orbifolds virtually fiber \cite{Agol}, allowing us to pass to a finite cover of $M$ that is a fibered 3-manifold. The fibers are punctured surfaces, $\Sigma$, whose punctures meet the cusps of $M$. Truncating the cusps of $M$ and then doubling over the truncated cusps glues the punctures of $\Sigma$ to the punctures in its copy, resulting in a closed surface (see Figure~\ref{fig:mfld.double}). It is this surface whose fundamental group will ultimately form a thin subgroup of $\Gamma$. In this section, we describe how to embed the fundamental group of this double into a non-uniform arithmetic subgroup of $\textnormal{SO}^+(4,1)$.

\subsection{The Folded Double}
\label{ch:folded.double}

We now define the folded double of a finite volume, non-compact, hyperbolic arithmetic 3-manifold. The definition uses the language of graphs of spaces, which allow us to construct useful geometric realizations corresponding to groups defined by graphs of groups. 

\begin{definition}
A \emph{graph of spaces} $(W,Y)$ is a graph $Y$ and a collection, $W$, of connected topological spaces associated to each vertex $v\in V(Y)$, $W_v$, and each edge $e=(v_0,v_1)\in E(Y)$, $W_e$, along with $\pi_1$-injective, continuous maps $\alpha_{(e,0)}$ and $\alpha_{(e,1)}$ from $W_e$ to the spaces associated to each vertex of $e$.
\end{definition}

A graph of spaces $(W,Y)$ has an associated geometric realization, $X$, formed by taking the disjoint union of the $W_v$ with the $W_e \times [0,1]$, then taking the quotient by the relation that glues $W_e\times\{0\}$ to the image of $\alpha_{(e,0)}$ in $W_{v_0}$ and gluing $W_e\times\{1\}$ to the image of  $\alpha_{(e,1)}$ in $W_{v_1}$. By \cite{Althoen}, when $X$ is connected, its fundamental group is simply $\pi_1(G,Y)$, where $(G,Y)$ is the graph of groups with vertex groups $G_v = \pi_1(W_v)$, edge groups $G_e = \pi_1(W_e)$, and monomorphisms induced by the $\pi_1$-injective, continuous maps in $(W,Y)$.

The graphs of spaces construction will be the basis for how we define representations of our arithmetic hyperbolic lattices. Let $M$ be a non-compact, fibered, hyperbolic arithmetic 3-manifold with finitely many cusps, $C_0, C_1,\ldots, C_{n-1}$. $M$ has a topological double $DM$ formed by making a copy of the manifold, truncating the cusps and identifying each boundary component with its counterpart in the copy. Because $M$ has the fiber surface $\Sigma$, $DM$ has a fiber subsurface $D\Sigma$, formed by doubling the fiber surface. Note that $D\Sigma$ is constructed by copying $\Sigma$ and gluing each puncture to the corresponding puncture in the copy, closing the surface. 

This doubling process can be recreated up to homeomorphism using a graph of spaces where the vertices are copies of the manifold with the cusps truncated, the edges are cusp cross-sections, and the edge maps are the identification of each cross-section with the corresponding boundary component in the truncated manifold. This gives a natural way of computing $\pi_1(DM)$ as the fundamental group of the graph of groups shown in Figure~\ref{fig:graph.double}, where $M'$ denotes the copy of $M$ and $C_i'$ is denotes copy of $C_i$ in $M'$. If we use the edge $t_0$ as the minimal spanning tree, we get that
\[
\pi_1(DM)\cong\langle \pi_1(M), \, \pi_1(M'), \, t_1,\ldots,t_{n-1} \rangle / R
\]
where the $t_i$ are the stable letters and $R$ is the normal subgroup generated by the relations:
\[
\{ c_0c_0'^{-1}, \, t_ic_it_i^{-1}c_i'^{-1} \,|\,  c_i \in \pi_1(C_i), c_i' \in \pi_1(C_i')\textrm{ for all } 0<i<n\}.
\]

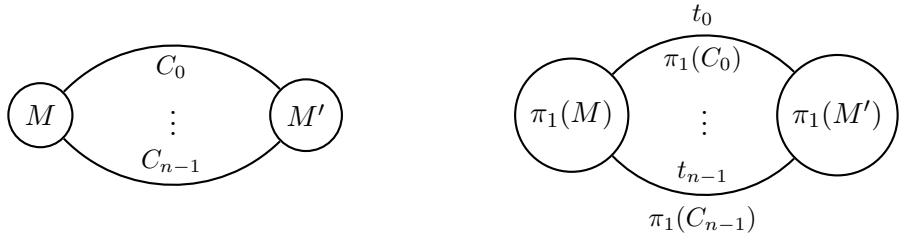
\begin{figure}[!ht] 
\centering
\begin{tikzpicture}[-,auto,node distance=100pt, thick,main node/.style={circle,draw,font=\sffamily}]

  \node[main node] (1) {$M$};
  \node[main node] (2) [right of=1] {$M'$};
  \node (3) [right of=1, node distance=50pt] {$\vdots$};

  \path[every node/.style={font=\sffamily\small}]
    (1) edge [bend left = 45] node[below] {$C_0$} (2)
        edge [bend right = 45] node[above] {$C_{n-1}$} (2);

  \node[main node] (4) [right of=2] {$\pi_1(M)$};
  \node[main node] (5) [right of=4] {$\pi_1(M')$};
  \node (6) [right of=4, node distance=50pt] {$\vdots$};

  \path[every node/.style={font=\sffamily\small}]
    (4) edge [bend left = 45] node[above] {$t_0$} node[below] {$\pi_1(C_{0})$}    (5)
        edge [bend right = 45] node[above] {$t_{n-1}$} node[below] {$\pi_1(C_{n-1})$} (5);

\end{tikzpicture}
\caption{The graph of spaces that forms $DM$, alongside the corresponding graph of groups that gives us $\pi_1(DM)$}
\label{fig:graph.double}
\end{figure}

\begin{figure}[!ht] 
\centering
\begin{tikzpicture}[-,auto,node distance=100pt, thick,main node/.style={circle,draw,font=\sffamily}]

            \draw[rounded corners=19pt](6.3,-1.17)--(5.53,-.8)--(4.3,-1.55)--(3.25,-.8)--(1.6,-1.5)--(0.1,0) -- (1.6,1.5)--(3.25,.75)--(4.3,1.6)--(5.53,.75)--(7,1.5)--(8.5,0) -- (7,-1.5)--(5.53,-.8)--(4.3,-1.55)--(3.25,-.8)--(2.125,-1.27);


            \draw[rounded corners=10pt](3.95,-.875)--(3.6,-.6)--(3.6,.6)--(4.3,1.2)--(5,.6)--(5,-.6)--(4.3,-1.15) -- (3.95,-.875);
            
            \draw (1.5,0.02) arc (175:315:.5cm and 0.25cm);
            \draw (2.25,-0.21) arc (-30:180:0.35cm and 0.15cm);

            \draw (6.3,0.02) arc (175:315:.5cm and 0.25cm);
            \draw (7.05,-0.21) arc (-30:180:0.35cm and 0.15cm);

            \draw[color = red, rounded corners=10pt](3.8,-.88)--(3.3,-.5)--(3.2,.4)--(4.3,1.27)--(5.3,.5)--(5.4,-.4)--(4.3,-1.26)--(3.8,-.88);
            \draw[color = red, rounded corners=10pt](3.725,-.925)--(3.15,-.5)--(1.5,-.7)--(1.1,0) -- (1.5,.7)--(3.15,.5)--(4.3,1.35)--(5.45,.5)--(7.1,.7)--(7.5,0) -- (7.1,-.7)--(5.45,-.5)--(4.3,-1.35)--(3.725,-.925);

            \node (a) at (1,.5) {$M$};
            \node (b) [right of=a, node distance=188pt] {$M$ };
            \node (sigma) [color = red, below of=a, node distance=25pt] {$\Sigma$ };
            \node (sigma2) [color = red, right of=sigma, node distance=188pt] {$\Sigma$ };
            \node (c1) at (4.3,.75) {$C_0$};
            \node (c2) [below of=c1, node distance=43.5pt] {$C_1$};

\end{tikzpicture}
\caption{Diagram of the subsurface $D\Sigma$ in $DM$}
\label{fig:mfld.double}
\end{figure}
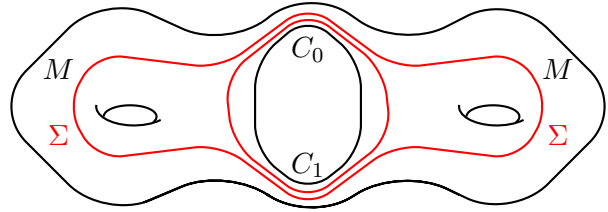

In order to embed $\pi_1(DM)$ into the fundamental group of a cusped arithmetic hyperbolic 4-orbifold we first further embed $\pi_1(DM)$ into the fundamental group of a ``self-intersecting manifold'' as an infinite-index subgroup.  

\begin{definition}
Let $M$ be a manifold with boundary. The \emph{folded double}, $\fold$, of $M$ is formed by taking the double $DM$ as defined in graph of spaces construction and identifying the two copies of $M$ that comprise the vertex set. This graph has one vertex $M$ and $n$ loops, whose edge spaces are the cusps of $M$.  
\end{definition}

From the graph of spaces description of the double, we see that since the edge spaces are $[0,1] \times S_i$ where $S_i$ is the cross-section of $C_i$, then after identifying vertices to form the folded double, we get a mapping torus in $\fold$ for each boundary component. Additionally, after folding, $\fold$ contains an immersed subsurface $F\Sigma$, formed by taking $D\Sigma$ and folding it back onto itself. Because we use the graph of spaces construction there is a self-intersecting torus that meets the punctures of $\Sigma$.

When we extend this folding operation to the fundamental group, we get a description of $\pi_1(\fold)$ as a modified HNN extension of $\pi_1(M)$ over multiple boundary components. As before we can explicitly compute this fundamental group using the corresponding graph of groups pictured in Figure~\ref{fig:graph.fold}. This tells us that $\pi_1(\fold)$ has the presentation below, where $[t_i,c_i]$ is the commutator. 
\[
\pi_1(\fold) = \langle \pi_1(M),t_0,t_1,\ldots, t_{n-1} \hspace{3pt} | \hspace{3pt} [t_i,c_i], \hspace{3pt} \forall  0\leq i < n, \hspace{3pt} c_i \in \pi_1(C_i) \rangle 
\]

\begin{figure}[!ht] 
\centering
\begin{tikzpicture}[-,auto,node distance=100pt, thick,main node/.style={circle,draw,font=\sffamily}]

  \node[main node] (1) {$M$};
  \node (2) [right of=1, node distance=25pt] {$\vdots$};

  \path[every node/.style={font=\sffamily\small}]
    (1) edge [out=20,in=50,looseness=12] node[above] {$C_0$} (1);
  \path[every node/.style={font=\sffamily\small}] 
    (1) edge [out=335,in=305,looseness=12] node[below] {$C_{n-1}$} (1);

  \node[main node] (4) [right of=2] {$\pi_1(M)$};
  \node (5) [right of=4, node distance=30pt] {$\vdots$};

  \path[every node/.style={font=\sffamily\small}]
    (4) edge [out=15,in=45,looseness=8] node[above] {$\pi_1(C_{0})$} (4);
  \path[every node/.style={font=\sffamily\small}] 
    (4) edge [out=335,in=305,looseness=8] node[below] {$\pi_1(C_{n-1})$} (4);

\end{tikzpicture}
\caption{The graph of spaces, $\fold$, and the graph of groups that gives us $\pi_1(\fold)$}
\label{fig:graph.fold}
\end{figure}
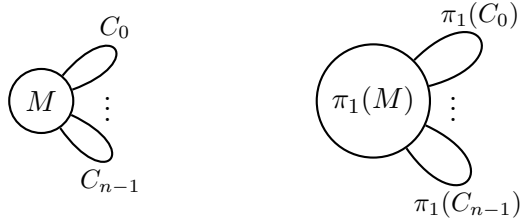

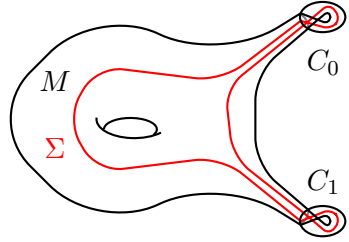
\begin{figure}[!ht] 
\centering
\begin{tikzpicture}[-,auto,node distance=100pt, thick,main node/.style={circle,draw,font=\sffamily}]

            \draw[rounded corners=25pt](4.2,-1.4)--(3.15,-.75)--(1.5,-1.5)--(0,0) -- (1.5,1.5)--(3.15,.75)--(4.2,1.4);

            \draw (1.5,0.02) arc (175:315:.5cm and 0.25cm);
            \draw (2.25,-0.21) arc (-30:180:0.35cm and 0.15cm);

            \draw[rounded corners=10pt](4.3,-1.2)--(3.6,-.6)--(3.6,.6)--(4.3,1.2);

            \draw[color = red, rounded corners=8pt](4.25,-1.25)--(3.3,-.5)--(3.2,.4)--(4.25,1.25);
            \draw[color = red, rounded corners=10pt](4.2,-1.35)--(3.15,-.5)--(1.5,-.7)--(1.1,0) -- (1.5,.7)--(3.15,.5)--(4.2,1.3);
            \draw[gray] (4.3,1.21) arc (0:60:0.2cm);
            \draw[color = red, rounded corners=3pt](4.25,1.25)--(4.58,1.53)--(4.73,1.33)--(4.6,1.19)--(4.2,1.3);
            \draw[gray] (4.3,-1.21) arc (0:-60:0.2cm);
            \draw[color = red, rounded corners=3pt](4.25,-1.25)--(4.58,-1.53)--(4.73,-1.33)--(4.6,-1.19)--(4.2,-1.35);

            \draw[] (4.49,1.35) ellipse (0.3 and 0.2);
            \draw[rounded corners=3pt] (4.29,1.2)--(4.59,1.45)--(4.59,1.25)--(4.2,1.4);
            \draw[] (4.49,-1.35) ellipse (0.3 and 0.2);
            \draw[rounded corners=3pt] (4.29,-1.2)--(4.59,-1.45)--(4.59,-1.25)--(4.2,-1.4);

            \node (a) at (.95,.5) {$M$};
            \node (sigma) [color = red, below of=a, node distance=25pt] {$\Sigma$ };
            \node (c1) at (4.5,.75) {$C_0$};
            \node (c2) [below of=c1, node distance=45pt] {$C_1$};

\end{tikzpicture}
\caption{Diagram of $\fold$ as it appears in $X$. We will construct a copy of $\fold$ such that it and the desired closed surface will be immersed in $X$, not embedded.
}
\label{fig:mani.fold}
\end{figure}

\begin{lemma}\label{lemma:graph.injection}
$\pi_1(\fold)$ has a surface subgroup. 
\end{lemma}
\begin{proof}
We know that $\pi_1(DM)$ has a surface subgroup, so we will construct an injective map $f:\pi_1(DM) \hookrightarrow \pi_1(FM)$. Refer to Figure~\ref{fig:graph.isomorphisms} for a graphical construction of this function. 

Let $(Y_0,G_0)$ be the graph of groups that we used to compute the fundamental group of $DM$. As previously mentioned this graph has a natural presentation of its fundamental group computed using the edge $\{t_0\}$ as the minimal spanning tree. This presentation is 
$\pi_1(Y_0,G_0) = (\pi_1(M)*_{\pi_1(C_0)}\pi_1(M))*\langle t_1,\ldots, t_{n-1}\rangle / t_ic_it_i^{-1}c_i'^{-1}$ 
for all $0<i<n$, $c_i \in \pi_1(C_i)$, and $c_i' \in \pi_1(C_i')$ where the $t_i$ are the stable letters. 

We can form a graph with same fundamental group as $FM$ by adding a new edge $e$ to $(Y_0,G_0)$, where the edge group of $e$ is $\pi_1(M)$, and the edge maps are both the identity function. This gives us the middle graph in Figure~\ref{fig:graph.isomorphisms}. Call this graph $(Y_1,G_1)$. Now $(Y_0,G_0)$ is a connected subgraph of $(Y_1,G_1)$, so there is a natural faithful representation of its fundamental group in $\pi_1(Y_1,G_1)$ induced by the inclusion map $(Y_0,G_0) \hookrightarrow (Y_1,G_1)$. In this case, the natural representation can be described using the presentation of $\pi_1(Y_1,G_1)$ where $t_0$ is the minimal spanning tree. Then $\pi_1(Y_1,G_1) = (\pi_1(M)*_{\pi_1(C_0)}\pi_1(M)) *\langle t_1,\ldots, t_{n-1}\rangle / (t_ic_it_i^{-1}c_i'^{-1})* e/eme^{-1}m'$ for all $m \in \pi_1(M)$ and $m' \in \pi_1(M)$, so the inclusion of $\pi_1(Y_0,G_0)$ as above into $\pi_1(Y_1,G_1)$ is a faithful representation. 

If we want to see that $\pi_1(Y_1,G_1)$ is isomorphic to $\pi_1(\fold)$, then we can instead let the edge $e$ make up our minimal spanning tree. This gives us $\pi_1(Y_1,G_1) = \langle \pi_1(M),t_0,t_1, \ldots, t_{n-1} \hspace{3pt} | \hspace{3pt} [t_i,c_i] \rangle$, which still has a subgroup that's isomorphic to $\pi_1(Y_0,G_0)$. Now we simply contract $e$ in $(Y_1,G_1)$, which identifies the two copies of $\pi_1(M)$, giving us the graph used to compute $\pi_1(\fold)$. This is the third graph in Figure~\ref{fig:graph.isomorphisms}. This graph has empty minimal spanning tree, so it simply has fundamental group $\pi_1(\fold) = \langle \pi_1(M),t_0,t_1, \ldots, t_{n-1} \hspace{3pt} | \hspace{3pt} [t_i,c_i] \rangle$ which is equal to the second presentation of $\pi_1(Y_1,G_1)$. 

Thus, there is an injective function $f:\pi_1(DM) = \pi_1(Y_0,G_0) \hookrightarrow \pi_1(Y_1,G_1) = \pi_1(\fold)$, given by the change in choice of minimal spanning tree for $\pi_1(Y_1,G_1)$. 
\end{proof}

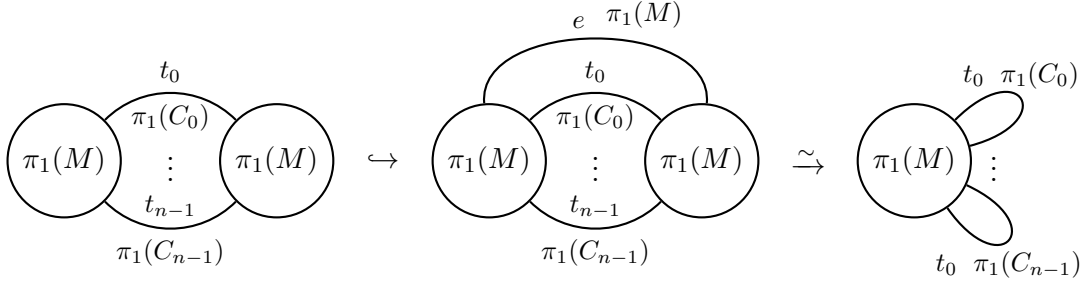
\begin{figure}[!ht] 
\centering
\begin{tikzpicture}[-,auto,node distance=80pt, thick,main node/.style={circle,draw,font=\sffamily,}]

  \node[main node] (1) {$\pi_1(M)$};
  \node[main node] (2) [right of=1] {$\pi_1(M)$};
  \node (3) [right of=1, node distance=40pt] {$\vdots$};

  \path[every node/.style={font=\sffamily\small}]
    (1) edge [bend left = 45] node[above] {$t_0$} node[below] {$\pi_1(C_{0})$}    (2)
    edge [bend right = 45] node[above] {$t_{n-1}$} node[below] {$\pi_1(C_{n-1})$} (2);

  \node (9) [right of=2, node distance=40pt] {$\hookrightarrow$};

  \node[main node] (4)  [right of=2] {$\pi_1(M)$};
  \node[main node] (5) [right of=4] {$\pi_1(M)$};
  \node (6) [right of=4, node distance=40pt] {$\vdots$};

  \path[every node/.style={font=\sffamily\small}]
    (4) edge [bend left = 45] node[above] {$t_0$} node[below] {$\pi_1(C_{0})$}    (5)
    edge [bend right = 45] node[above] {$t_{n-1}$} node[below] {$\pi_1(C_{n-1}) $} (5)
    (4) edge [bend left = 95] node[above left] {$e$} node[above right] {$\pi_1(M)$}    (5)
    ;

\node (9) [right of=5, node distance=40pt] {$\xrightarrow{\sim}$};

  \node[main node] (7) [right of=5] {$\pi_1(M)$};
  \node (8) [right of=7, node distance=30pt] {$\vdots$};

  \path[every node/.style={font=\sffamily\small}]
    (7) edge [out=15,in=45,looseness=8] node[above] {$t_0 \hspace{6pt} \pi_1(C_{0})$} (7);
  \path[every node/.style={font=\sffamily\small}] 
    (7) edge [out=335,in=305,looseness=8] node[below] {$t_0 \hspace{6pt} \pi_1(C_{n-1})$} (7);

\end{tikzpicture}
\caption{An equivalent way to construct the folded double that gives us an injective map on fundamental groups. First add an edge $e$ to the double such that $e$ has edge group $\pi_1(M)$ and edge maps that are simply the identity on both copies of $\pi_1(M)$. Then contract that edge to identify the copies of $\pi_1(M)$.}
\label{fig:graph.isomorphisms}
\end{figure}

\subsection{Representations in the Fundamental Groups of 4-Manifolds and Orbifolds}
Now that we have abstractly described the folded double of a cusped, arithmetic, hyperbolic 3-manifold $M$, we construct a representation of $\pi_1(FM)$ into the fundamental group, $\Gamma$, of a cusped, arithmetic, hyperbolic 4-manifold $X$ and use its action on $\mathbb{H}^4$ to show that this representation is faithful.

We first define some of the concepts that we use to characterize cusps and their associated subgroups in $\pi_1(M)$. Let $C$ be a cusp of $M$. A \emph{cusp point} associated to $C$ is a point on the boundary of $\h^n$ that is approached by the preimage of $C$ in some fundamental domain for the action of $\pi_1(M)$ on $\h^n$. We call the stabilizer of a cusp point in $\pi_1(M)$ of a \emph{cusp subgroup}. $M$ is orientable, so its cusps have flat 2-torus cross-sections. Thus, each of these subgroups is abstractly isomorphic to $\mathbb{Z}^2$ and generated by two parabolic isometries.

Let $S$ be a cross-section of $C$, so that $C$ is homeomorphic to $(0,\infty) \times S$. As $C$ has finite volume, the diameter of the slice $\{t\}\times S$ decreases exponentially as $t\to\infty$ in the cusp $C$. We use the notion of cusp depth to keep track of relative position within each cusp.

\begin{definition}
    The \emph{cusp depth} of a cusp cross-section with $(n-1)$-dimensional volume $V$ is $\frac{1}{V}$.  
\end{definition}

As we move towards infinity along the real axis of the cusp, the cusp depth goes to infinity as well. Thus, a larger cusp depth corresponds to taking a smaller cross-section. We can use this notion of relative location within the cusp to determine the parabolic isometries that will define our representation of $\pi_1(\fold)$.

In order to build our desired representation we first realize the cusp subgroups of $\pi_1(M)$ as subgroups of the cusp subgroups of $\Gamma$ with certain nice geometric properties. First, let $G$ be the image of $\varphi$, where $\varphi: \pi_1(M) \hookrightarrow \Gamma$ is the representation described in Lemma~\ref{lemma:groups.and.qf}. $\Gamma$ acts on $\h^4$ with $G$ preserving a copy of $\h^3 \subset \h^4$ as $G$ is contained in a copy of $\SOfZ$ for $f$ a form of signature $(3,1)$. We denote this copy of $\h^3$ by $H_G$.

Let $P_0,\ldots, P_{n-1}$ be the collection of cusp subgroups of $G$. Each cusp subgroup is a maximal parabolic subgroup in $G$ isomorphic to $\mathbb{Z}^2$, generated by two parabolic elements that preserve a family of three-dimensional horospheres in $\h^4$ intersecting the boundary of $H_G$. For each $P_i$, $i \in \{0,\ldots,n-1\}$ choose one of these horospheres and denote it $h_i^3$. The intersection of this horosphere and $H_G$ is a two-dimensional horosphere denoted $h_i^2$. 

If we identify $M$ with its image in $X$ under the map induced by $\varphi$, then $M$ is a totally geodesic immersed submanifold. By passing to finite covers of $M$ and $X$ we can ensure that $M$ is embedded in $X$, $X$ has 3-torus cusps, and each cusp of $M$ lies in some distinct cusp of $X$ \cite{MRS}. That is, the cusps do not collide at infinity. Then for each cusp $C_i$ of $M$ there is some parabolic element $p_i$ in $\Gamma$ such that $\langle P_i, p_i \rangle$ is a maximal cusp subgroup for the cusp of $X$ containing $C_i$. Hence, $p_i$ commutes with $P_i$ and preserves $h_i^3$. However, instead of preserving $H_G$, $p_i$ translates $H_G$ to some new copy of $\h^3$. Thus, $p_i$ moves $h_i^2$ a fixed positive distance off of itself. As we take higher powers of $p_i$, this distance goes to infinity. We can now use the action of these parabolic elements to build a representation of $\pi_1(\fold)$ in $\Gamma$ and prove the following.

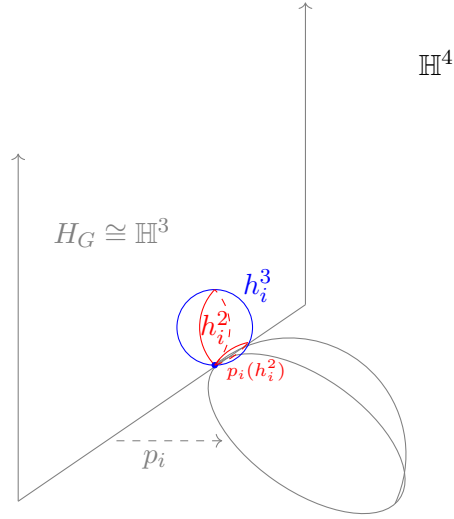
\begin{figure}[!ht] 
\centering
\begin{tikzpicture}

            \draw[thin] (6.5,4.5) node[left]{$\h^4$};

            \draw[color = gray, thin] (2.75,2.25) node[left]{$H_G \cong \h^3$};
            \draw[color = gray, thin] (4.4,1.3)--(.6,-1.3);
            \draw[color = gray, thin, ->] (4.4,1.3) -- (4.4,5.3);
            \draw[color = gray, thin, ->] (.6,-1.3) -- (.6,3.3);

            \draw[color = gray, thin] (3.2,.5) arc (130:-25:1.54);
            \draw[color = gray, thin, rotate=-35] (3.85,2.2) ellipse (1.5 and 0.75);

            \draw[color=blue] (3.2,1) circle (.5cm);
            \tkzDefPoint(3.2,1.5){top2}
            \tkzDefPoint(3.2,.5){boundary2}
            \tkzDefPoint(3.4,1){front2}
            \tkzDefPoint(3,1){back2}
            \arcThroughThreePoints[thin, dashed, color=red]{boundary2}{front2}{top2};
            \arcThroughThreePoints[thin, color=red]{top2}{back2}{boundary2};
            \tkzDrawPoints[color=blue](boundary2)
            \draw[color = blue] (4.1,1.2) node[above left]{$h_i^3$};
            \draw[color = red] (3.2,1) node[]{$h_i^2$};

            \tkzDefPoint(3.64,.8){Ptop}
            \tkzDefPoint(3.4,.7){Pback}
            \tkzDefPoint(3.55,.7){Pfront}
            \arcThroughThreePoints[thin, color=red]{Ptop}{Pback}{boundary2};
            \arcThroughThreePoints[thin, dashed, color=red]{boundary2}{Pfront}{Ptop};

            \draw[color = red] (3.75,0.15) node[above]{\tiny $p_i(h_i^2)$};

            \draw[dashed, ->][color=gray] (1.9,-0.5) -- (3.3,-0.5);
            \draw[color = gray] (2.7,-1) node[above left]{$p_{i}$};

\end{tikzpicture}
\caption{A diagram showing how the parabolic element $p_i$ translates the plane $H_G$. The red curve in $p_i(H_G)$ is $p_i(h_i^2)$. Note that it remains in $h^3_i$.}
\label{fig:parabolic}
\end{figure}

\begin{theorem}\label{main.theorem}
Let $X$, $M$, and $\Gamma$ be as above. Then $\Gamma$ contains a subgroup that is isomorphic to $\pi_1(\fold)$.
\end{theorem}

We first define a homomorphism $\varrho:\pi_1(\fold)\to\Gamma$ via where it sends generators of $\pi_1(\fold)$:
\[
\varrho(x):=\begin{cases}
\varphi(x) & x\in\pi_1(M)\\
p_i & x=t_i.
\end{cases}
\]
Observe that as the parabolic element $p_i$ commutes with the entire cusp subgroup $P_i$ in $\operatorname{SO}^+(4,1)$, $\varrho$ defines a homomorphism of $\pi_1(\fold)$. We will show that, after possibly raising each $p_i$ to a higher power, $\varrho$ is faithful. 

To show faithfulness, for every nontrivial word $\omega$ in the generators of $\pi_1(\fold)$, we will construct some $x\in\h^4$ with the property that $\varrho(\omega)x\neq x$. The argument rests on a ping-pong type construction of a broken (i.e.\ piecewise) geodesic in $\h^4$ through $x$ and $\varrho(\omega)x$ with certain nice geometric properties.

First, for a broken geodesic $\gamma$ comprised of geodesic segments $\gamma_1,\gamma_2,\gamma_3,\ldots$ that all meet at angles greater than $\frac{\pi}{2}$, there is a constant $D_0$ such that if each segment of $\gamma$ is of length at least $D_0$, then $\gamma$ is quasi-geodesic. Moreover, the quasi-geodesic constants for $\gamma$ can be chosen independently of the length of each segment (see \cite{B+H}*{p. 407}). For a quasi-geodesic, there is some constant $r$ determined only by the quasi-constants such that it cannot return to a point after a distance greater than $r$ away from that point. Fix some $D>\max(2r,D_0)$. We will find an $x\in\h^4$ and a broken geodesic, $\gamma_\omega$, of the above form comprised of geodesic segments of length at least $D$ such that $x$ and $\rho(\omega)x$ are contained in distinct geodesic segments. As a piecewise geodesic can only self intersect along non-consecutive segments, if $\gamma_\omega$ contained a loop, then some geodesic segments $\gamma_i$ and $\gamma_j$ intersect with $j>i+1$. This means one must traverse at least the length of $\gamma_{i+1}$ to get from $\gamma_i$ to $\gamma_j$, and as $D>2r$, this is impossible. In particular, $\gamma_\omega$ does not self intersect and so $x\neq\rho(\omega)x$.

Building the broken geodesic of the above form rests on a normal form for elements of $\pi_1(\fold)$ coming from its graph of groups description. Note that from the presentation of $\pi_1(\fold)$, there exists some $m_i\in\pi_1(M)$ (possibly the identity), non-zero $k_i\in\mathbb{Z}$, and $r_i\in\{0,\ldots,n-1\}$ such that
\[
\omega = m_1 t_{r_1}^{k_1}m_2 t_{r_2}^{k_2}\ldots m_\ell.
\]
We shall say that a word of this form has length $2\ell-1$ (note that this is often referred to as the syllable length of the word). Now, if $c_{r_i}$ is an element of the cusp subgroup $P_{r_i}$, then any subword of the form $t_{r_i}^{k_{i-1}}c_{r_i}t_{r_i}^{k_i}$ in $\omega$ can be shortened to $t_{r_i}^{k_{i-1}+k_i}c_{r_i}$. One can perform this reduction multiple times. Then after a finite number of iterations, we can assume that $\omega$ is of the above form and has no subwords of the form $t_{r_i}^{k_{i-1}}c_{r_i}t_{r_i}^{k_i}$. By Britton's Lemma \cite{L+S}, we have that $\omega$ in the above reduced form is the identity in $\pi_1(\fold)$ if and only if $\ell=1$, $m_1=1$ and $k_1=0$. This normal form for elements of $\pi_1(\fold)$ gives natural way to prove Theorem~\ref{main.theorem} by inducting on $\ell$.

\begin{proof}[Proof of Theorem~\ref{main.theorem}]
Let $\omega=m_1 t_{r_1}^{k_1}m_2 t_{r_2}^{k_2}\ldots m_\ell\in\pi_1(\fold)$ be a nontrivial reduced word as above. We will construct the quasi-geodesic $\gamma_\omega$ by inducting on $\ell$.

First, as it is instructive to our argument, we can directly show that $\varrho$ is injective on words with $\ell=1$ or $2$. For $\ell=1$, this directly follows from injectivity of $\varphi$, so that, identifying $\pi_1(M)$ with $\varphi(\pi_1(M))$, we have that if $1\neq \omega=m_1 t_{r_1}^{k_1}$, then $\varrho(\omega)= m_1p_{r_1}^{k_1}$ and this is nontrivial as $p_{r_1}\notin\pi_1(M)$. When $\ell=2$, we rewrite $\omega= m_1 t_{r_1}^{k_1}m_2$ as
\[
\omega = (m_1 p_{r_1}^{k_1})m_2(m_1 p_{r_1}^{k_1})^{-1}\cdot (m_1)p_{r_1}^{k_1}(m_1)^{-1}\cdot m_1.
\]
Now, fix an $x\in H_G$ and consider the action of $\varrho(\omega)$ on $x$ as we iterate through the three subwords above from right to left. First, $m_1$ acts on $x$ (possibly fixing it), sending it to a point in $H_G$. The subword $(m_1)p_{r_1}^{k_1}(m_1)^{-1}$ is a parabolic isometry fixing the point $m_1 y_{r_1}\in\partial H_G$, where $y_{r_1}$ is the fixed point of $p_{r_1}$. Hence, $(m_1)p_{r_1}^{k_1}(m_1)^{-1}$ moves $H_G$ off of itself to a new copy of $\h^3$ inside $\h^4$, meeting $H_G$ only at the boundary point $m_1 y_{r_1}$, see Figure~\ref{fig:parabolic}. Finally, the subword $(m_1 p_{r_1}^{k_1})m_2(m_1 p_{r_1}^{k_1})^{-1}$ restricts to an isometry of $(m_1)p_{r_1}^{k_1}H_G=(m_1)p_{r_1}^{k_1}(m_1^{-1})H_G$. Therefore, following $x$ under the action of each subword of $\varrho(\omega)$, we see that $x\in H_G$ whereas $\varrho(\omega)x\in (m_1)p_{r_1}^{k_1}H_G$, and as these two copies of $\h^3$ inside $\h^4$ intersect only at the boundary point $m_1 y_{r_1}$, this is enough to see that $\varrho(\omega)x\neq x$.
 
To show $\varrho$ is injective on $\omega$ with $\ell>2$, we will use an argument similar to the above to construct the broken geodesic $\gamma_\omega$ comprised of geodesic segments of length at least $D$ meeting at angles greater than $\frac{\pi}{2}$ through $x$ and $\varrho(\omega)x$ for an arbitrary $x\in H_G$. To do so, it is first necessary to modify our representation $\varrho$ slightly. If we lift the thin part of $X$ to $\h^4$ then we get a family of $3$-dimensional horoballs, each of which is centered at a boundary point or cusp point of $X$. These cusp points are exactly the fixed points of parabolic elements in $X$. $X$ has a finite number of cusps, so if we choose an appropriately large cusp depth for each cusp in the thin part of $X$, we can ensure that these horoballs are an arbitrarily large distance apart from one another. Let $\mathcal{B}$ be a family of horoballs with large enough cusp depth such that each horoball is at least a distance $D$ from every other horoball. Let $\mathcal{S}$ be the collection of horospheres bounding the horoballs in $\mathcal{B}$. Finally, for each parabolic generator $p_i$, with fixed point $y_i$, let $s_i$ be a horosphere in $\mathcal{S}$ centered at $y_i$. Replace $p_i$ with a sufficiently high power $p_i^{j_i}$ such that $p_i^{j_i}$ moves $s_i\cap H_G$ a distance at least $D$ off of itself. We then modify the definition of $\varrho$ to instead send $t_i$ to $p_i^{j_i}$, rather than $p_i$. Note that replacing the $p_i$ in this way does not affect the previous part of this proof.

We will show that there is a broken geodesic $\gamma_\omega$ through $x$ and $\varrho(\omega)x$ that satisfies the following properties:
\begin{enumerate} 
    \item $\gamma_\omega$ has segments of length at least $D$ meeting at angles greater than $\pi/2$
    \item $\gamma_\omega$ consists of $2\ell-1$ segments
    \item The third to last segment of $\gamma_\omega$, which is $\gamma_{2\ell-3}$, is contained in the copy of $\h^3$ preserved by $(m_1 \ldots p_{r_{\ell-2}}^{k_{\ell-2}} ) m_{\ell-1} (m_1 \ldots p_{r_{\ell-2}}^{k_{\ell-2}})^{-1}$ and it intersects the horoball in $\mathcal{B}$ preserved by $(m_1 \ldots m_{\ell-1}) p_{r_{\ell-1}}^{k_{\ell-1}} (m_1 \ldots m_{\ell-1})^{-1} $ at a right angle. 
\end{enumerate}
We use $\ell=2$ as the base step for the induction, in which case, the third condition becomes ``$\gamma_{1}$, is contained in the hyperbolic plane preserved by $m_1$ and it intersects the horoball in $\mathcal{B}$ preserved by $(m_1) p_{r_1}^{k_1} (m_1)^{-1}$ at a right angle.''

To form this broken geodesic consider the $3$-dimensional horosphere $h^3$ in $\mathcal{S}$ centered at $m_1 (y_{r_1})$. Take the geodesic containing $x$ that approaches the boundary point $m_1 (y_{r_1})$ and the geodesic containing $\varrho(\omega)x$ that approaches $m_1 (y_{r_1})$. Truncate these geodesics at the point where they intersect $h^3$. Note that they intersect $h^3$ in a right angle because they contain the center of this horosphere. This gives two geodesic rays, $\gamma_1$ lying in $H_G$ and $\gamma_3$ lying in $(m_1) p_{r_1}^{k_1}H_g$. Now take the geodesic segment that connects the terminal point of $\gamma_1$ and the starting point of $\gamma_3$. This segment is $\gamma_2$ and lies in the horoball in $\mathcal{B}$ that $h^3$ bounds.

This broken geodesic satisfies properties 2 and 3 by construction. To show that it satisfies the conditions the first property note that $\gamma_1$ and $\gamma_3$ are infinitely long. Also, $\gamma_2$ connects the two horospheres $h^3 \cap H_G$ and $h^3 \cap (m_1) p_{r_1}^{k_1} (m_1)^{-1} H_g$ to each other. By how the powers $p_{r_i}^{j_i}$ were chosen, these sets are at least $D$ away from each other so $\gamma_2$ has length at least $D$. Additionally, $\gamma_2$ lies entirely in the horoball bounded by $h^3$, and $\gamma_1$ and $\gamma_3$ meet the horoball at a right angle, so $\gamma_2$ meets $\gamma_1$ and $\gamma_3$ at an angle greater than $\frac{\pi}{2}$, see Figure~\ref{fig:basegeodesic}. Thus, we have our desired broken geodesic. 

\begin{figure}[h!] 
\centering
\begin{tikzpicture}

            \draw[thin] (-4,5) node[left]{$\h^4$};

            \draw[color = gray, thin] (-.6,1.3)--(-4.4,-1.3);
            \draw[color = gray, thin, ->] (-4.4,-1.3) -- (-4.4,3.3);
            \draw[color = gray, thin, ->] (-.6,1.3) -- (-.6,5.3);
            \draw[color = gray, thin] (-3.5,1.8) node[left]{$H_G$};

            \draw[color = gray, thin] (4.4,1.3)--(.6,-1.3);
            \draw[color = gray, thin, ->] (4.4,1.3) -- (4.4,5.3);
            \draw[color = gray, thin, ->] (.6,-1.3) -- (.6,3.3);

            \tkzDefPoint(0,5.5){infty}
            \tkzDrawPoints[color = blue](infty)
            \tkzLabelPoint[above right, color = blue](infty){$\infty = m_1(y_{r_1})$};
            \draw[color = blue] (-.6,4.5)--(-4.4,2.5);
            \draw[color = blue] (4.4,4.5)--(.6,2.5);
            \draw[color=blue] (4.4,4.5) -- (-.6,4.5);
            \draw[color=blue] (-4.4,2.5) -- (.6,2.5);
            \tkzDefPoint(1.3,2.9){h3}
            \tkzLabelPoint[above, color = blue](h3){$h^3$};

            \draw[color=blue] (-3.2,0) circle (.5cm);
            \tkzDefPoint(-3.2,.5){top}
            \tkzDefPoint(-3.2,-.5){boundary}
            \tkzDefPoint(-3.4,0){front}
            \tkzDefPoint(-3.0,0){back}
            \arcThroughThreePoints[thin, color=blue]{top}{front}{boundary};
            \arcThroughThreePoints[thin, dashed,color=blue]{boundary}{back}{top};
            \tkzDrawPoints[color=blue](boundary)
            \tkzLabelPoint[below, color = blue](boundary){$y_{r_1}$};

            \draw[dashed, ->][color=gray] (-2.3,.75) -- (2.7,.75);
            \draw[color = gray] (.75,0) node[above left]{$(m_1)p_{r_1}^{k_1}(m_1)^{-1}$};

            \tkzDefPoint(-1.5,1.75){x}
            \tkzDrawPoints[color=red](x)
            \tkzLabelPoint[right](x){$x$};
            \draw[->, color = red] (-1.49,4.02)--(-1.49,1.3);
            \draw[->, color = red] (3.2,3.86)--(3.2,1.0);
            \draw[color = red] (-1.5,4.02) arc (140:36.5:3);

            \tkzDefPoint(-2.05,3.2){gamma1}
            \tkzLabelPoint[right, color = red](gamma1){$\gamma_1$};

            \tkzDefPoint(0.5,4.8){gamma2}
            \tkzLabelPoint[right, color = red](gamma2){$\gamma_2$};

            \tkzDefPoint(3.2,2.8){gamma3}
            \tkzLabelPoint[right, color = red](gamma3){$\gamma_3$};

            \tkzDefPoint(3.2,1.5){wx}
            \tkzDrawPoints[color=red](wx)
            \tkzLabelPoint[right](wx){$\varrho(\omega(x))$};

\end{tikzpicture}
\caption{The broken geodesic in the base case. $h^3$ is the horosphere centered at infinity.}
\label{fig:basegeodesic}
\end{figure}
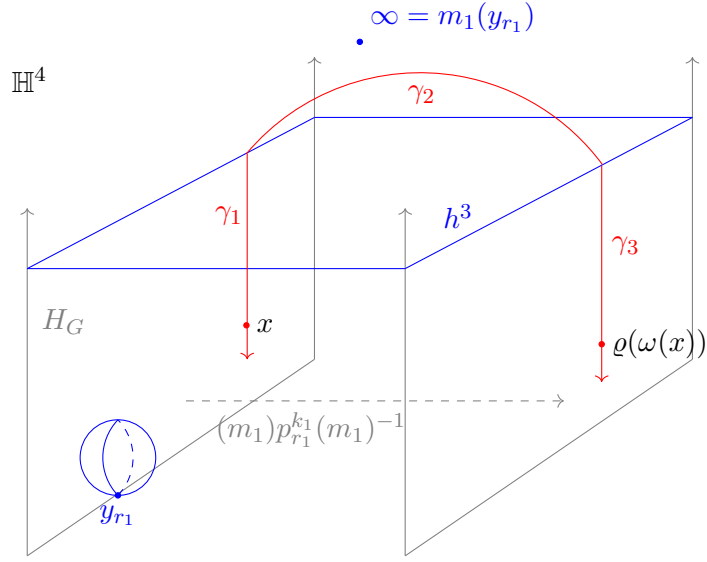

Finally, we induct and consider $\omega_1\in\pi_1(\fold)$ with $\varrho(\omega_1) = m_1  p_{r_1}^{k_1}  \ldots m_{\ell}  p_{r_{\ell}}^{k_{\ell}}  m_{\ell+1}$ of length $\ell+1>2$. Consider the subword $\varrho(\omega_0) = m_1  p_{r_1}^{k_1}  m_2  p_{r_2}^{k_2}\ldots m_{\ell-1}  p_{r_{\ell-1}}^{k_{\ell-1}}  m_{\ell}$ of length $\ell$. We have $\varrho(\omega_1) = \varrho(\omega_0) \cdot p_{r_{\ell}}^{k_{\ell}}m_{\ell+1}$ which we expand into three subwords similar to the $\ell=2$ argument:
\[
\varrho(\omega_1) = \varrho(\omega_0)  p_{r_{\ell}}^{k_{\ell}}   m_{\ell+1} (\varrho(\omega_0)  p_{r_{\ell}}^{k_{\ell}}  )^{-1} 
\cdot (\varrho(\omega_0)) p_{r_{\ell}}^{k_{\ell}} (\varrho(\omega_0))^{-1} 
\cdot \varrho(\omega_0).
\]
We then iterate through the subword from right to left. First, $\varrho(\omega_0)$ moves $H_G$ to $\varrho(\omega_0) H_G$. Next, $\varrho(\omega_0) p_{r_{\ell}}^{k_{\ell}} \varrho(\omega_0)^{-1} $ is a parabolic element fixing a point on the boundary of $\varrho(\omega_0) H_G$, and translates $\varrho(\omega_0) H_G$ to  $\varrho(\omega_0) p_{r_{\ell}}^{k_{\ell}} H_G$. Finally, $\varrho(\omega_0)  p_{r_{\ell}}^{k_{\ell}}   m_{\ell+1} (\varrho(\omega_0) p_{r_{\ell}}^{k_{\ell}})^{-1}$ is an isometry of $\varrho(\omega_0) p_{r_{\ell}}^{k_{\ell}} (H_G)$. Thus, if $x$ is in $H_G$, $\varrho(\omega_1)$ moves $x$ from $H_G$ to $\varrho(\omega_0) p_{r_{\ell}}^{k_{\ell}} H_G$. We use these copies of $\h^3$ to form our broken geodesic. 

Assume that $\varrho(\omega_0)$ has some broken geodesic $\gamma$ that satisfies the three conditions above. We form a new broken geodesic $\gamma'$ by removing the last two segments of $\gamma$ and adding 4 new segments. That is $\gamma'_j = \gamma_j$ for all $j\leq 2\ell-3$. Let $y$ be the fixed point of the element $(m_1 \ldots m_{\ell-1}) p_{r_{\ell-1}}^{k_{\ell-1}} (m_1 \ldots m_{\ell-1})^{-1}$ and $y'$ be the fixed point of $\varrho(\omega_0) p_{r_{\ell}}^{k_{\ell}} \varrho(\omega_0)^{-1}$. By the inductive hypothesis, the segment $\gamma_{2\ell-3}$ meets the horoball in $\mathcal{B}$ centered at $y$ at a right angle. Now take the geodesic that runs from $y$ to $y'$. Truncate it on both ends where it intersects the horoballs in $\mathcal{B}$ centered at $y$ and $y'$ and call this segment $\gamma'_{2\ell-1}$. Connect it to $\gamma_{2\ell-3}$ via the geodesic segment that lies in the horoball centered at $y$, which we call $\gamma'_{2\ell-2}$. Finally, consider the geodesic that runs from $y'$ to $\varrho(\omega_1(x))$. Truncate it where it meets the horoball centered at $y'$ to form $\gamma'_{2\ell+1}$. Connect it to $\gamma'_{2\ell-1}$ via the geodesic segment lying in the horoball centered at $y'$, which we call $\gamma'_{2\ell}$.

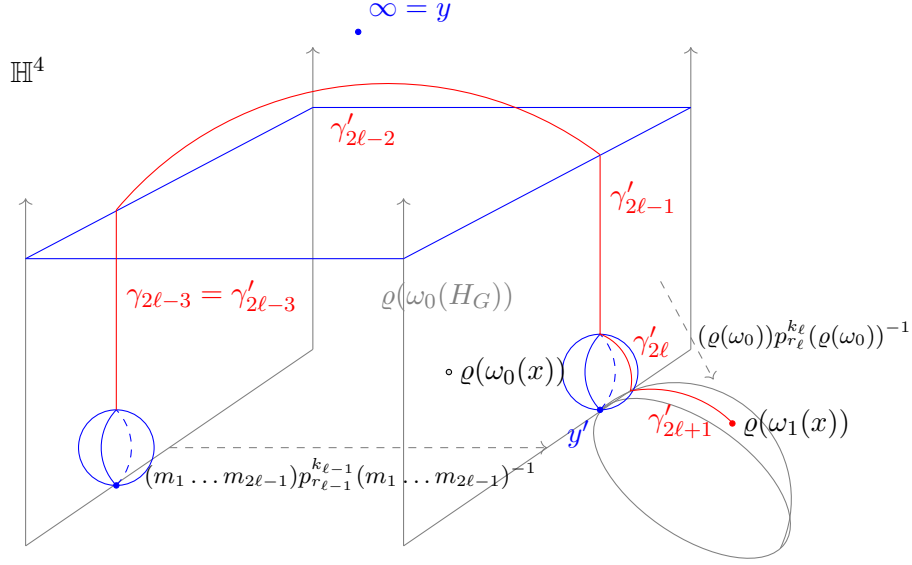
\begin{figure}[h!] 
\centering
\begin{tikzpicture}

            \draw[thin] (-4,5) node[left]{$\h^4$};

            \draw[color = gray, thin] (-.6,1.3)--(-4.4,-1.3);
            \draw[color = gray, thin, ->] (-4.4,-1.3) -- (-4.4,3.3);
            \draw[color = gray, thin, ->] (-.6,1.3) -- (-.6,5.3);
            \draw[color = gray, thin] (2.2,2) node[left]{$\varrho(\omega_0(H_G))$};

            \draw[color = gray, thin] (4.4,1.3)--(.6,-1.3);
            \draw[color = gray, thin, ->] (4.4,1.3) -- (4.4,5.3);
            \draw[color = gray, thin, ->] (.6,-1.3) -- (.6,3.3);

            \draw[color = gray, thin] (3.2,.5) arc (130:-25:1.54);
            \draw[color = gray, thin, rotate=-35] (3.85,2.2) ellipse (1.5 and 0.75);

            \tkzDefPoint(0,5.5){infty}
            \tkzDrawPoints[color = blue](infty)
            \tkzLabelPoint[above right, color = blue](infty){$\infty = y$};
            \draw[color = blue] (-.6,4.5)--(-4.4,2.5);
            \draw[color = blue] (4.4,4.5)--(.6,2.5);
            \draw[color=blue] (4.4,4.5) -- (-.6,4.5);
            \draw[color=blue] (-4.4,2.5) -- (.6,2.5);

            \draw[color=blue] (-3.2,0) circle (.5cm);
            \tkzDefPoint(-3.2,.5){top}
            \tkzDefPoint(-3.2,-.5){boundary}
            \tkzDefPoint(-3.4,0){front}
            \tkzDefPoint(-3.0,0){back}
            \arcThroughThreePoints[thin, color=blue]{top}{front}{boundary};
            \arcThroughThreePoints[thin, dashed,color=blue]{boundary}{back}{top};
            \tkzDrawPoints[color=blue](boundary)

            \draw[color=blue] (3.2,1) circle (.5cm);
            \tkzDefPoint(3.2,1.5){top2}
            \tkzDefPoint(3.2,.5){boundary2}
            \tkzDefPoint(3.4,1){front2}
            \tkzDefPoint(3,1){back2}
            \arcThroughThreePoints[thin, dashed, color=blue]{boundary2}{front2}{top2};
            \arcThroughThreePoints[thin, color=blue]{top2}{back2}{boundary2};
            \tkzDrawPoints[color=blue](boundary2)
            \tkzLabelPoint[below left, color = blue](boundary2){$y'$};

            \draw[dashed, ->][color=gray] (-2.5,0) -- (2.5,0);
            \draw[] (2.5,0) node[below left]{\footnotesize $(m_1 \ldots m_{2\ell-1}) p_{r_{\ell-1}}^{k_{\ell-1}} (m_1 \ldots m_{2\ell-1})^{-1}$};

            \draw[dashed, ->][color=gray] (4,2.2) -- (4.7,.9);
            \draw[] (4.35,1.5) node[right]{\footnotesize $(\varrho(\omega_0)) p_{r_{\ell}}^{k_{\ell}} (\varrho(\omega_0))^{-1}$};

            \tkzDefPoint(1.2,1){x}
            \tkzDrawPoints[](x)
            \tkzLabelPoint[right](x){$\varrho(\omega_0(x))$};

            \draw[color = red] (-3.2,.5)--(-3.2,3.145);
            \draw[color = red] (3.2,1.5)--(3.2,3.86);
            \draw[color = red] (-3.2,3.145) arc (140:53:4.68);
            \draw[color = red] (3.6,.75) arc (100:45:1.54);
            \draw[color = red] (3.6,.75) arc (-13:70:.65);

            \tkzDefPoint(4.95,.32){wx}
            \tkzDrawPoints[color=red](wx)
            \tkzLabelPoint[right](wx){$\varrho(\omega_1(x))$};


            \tkzDefPoint(-3.2,2){gamma1}
            \tkzLabelPoint[right, color = red](gamma1){$\gamma_{2\ell-3} = \gamma'_{2\ell-3}$};
            
            \tkzDefPoint(-.5,4.2){gamma2}
            \tkzLabelPoint[right, color = red](gamma2){$\gamma'_{2\ell-2}$};

            \tkzDefPoint(3.2,3.3){gamma3}
            \tkzLabelPoint[right, color = red](gamma3){$\gamma'_{2\ell-1}$};

            \tkzDefPoint(3.5,1.4){gamma4}
            \tkzLabelPoint[right, color = red](gamma4){$\gamma'_{2\ell}$};

            \tkzDefPoint(3.7,.35){gamma5}
            \tkzLabelPoint[right, color = red](gamma5){$\gamma'_{2\ell+1}$};

\end{tikzpicture}
\caption{The new segments added to $\gamma$ to form $\gamma'$. Note that $\varrho(\omega_0(x))$ is bypassed completely by $\gamma'$.}
\label{fig:qi.embedding}
\end{figure}

We have defined all the geodesic segments in $\gamma'$, and we claim that $\gamma_\omega=\gamma'$ is the desired broken geodesic. First, there are $2\ell+1$ total segments comprising $\gamma'$, so $\gamma'$ clearly satisfies the second required condition. Next, $y$ is the fixed point of $m_1 \ldots m_{\ell-1} p_{r_{\ell-1}}^{k_{\ell-1}} (m_1 \ldots m_{\ell-1})^{-1}$, a parabolic subword of $\varrho(\omega_0)$, and translates $m_1 \ldots m_{\ell-1}H_G$ to $m_1 \ldots m_{\ell-1} p_{r_{\ell-1}}^{k_{\ell-1}} H_G = \varrho(\omega_0)(H_G)$, so $y$ lies in $\varrho(\omega_0)(H_G)$. Also, $y'$ is the fixed point of $(\varrho(\omega_0)) p_{r_{\ell}}^{k_{\ell}} (\varrho(\omega_0))^{-1}$, so it lies in $\varrho(\omega_0)(H_G)$ as well. The third to last segment $\gamma'_{2\ell-1}$ is contained in the geodesic running from $y$ to $y'$. Thus, it is contained in $\varrho(\omega_0)H_G$, and intersects the horoball in $\mathcal{B}$ preserved by  $(\varrho(\omega_0)) p_{r_{\ell}}^{k_{\ell}} (\varrho(\omega_0))^{-1}$ at a right angle. Hence, $y'$ satisfies the third required condition.

We check the first condition for the last three segments comprising $\gamma'$ as in the base case. The segments $\gamma_{2\ell-3}'$, $\gamma_{2\ell-1}'$, and $\gamma_{2\ell+1}'$ meet horoballs at right angles and the segments in between them $\gamma_{2\ell-2}'$ and $\gamma_{2\ell}'$ lie inside these horoballs, so the angles in our broken geodesic are all greater than $\pi/2$. As for the length of the four new geodesics, $\gamma_{2\ell+1}'$ is infinitely long and $\gamma_{2\ell-1}'$ connects two horoballs in $\mathcal{B}$, so each of the segments is at least $D$ long. Call the horosphere in $\mathcal{S}$ centered at $y$ $h_y$ and the horosphere centered at $y'$ $h_{y'}$. The segment $\gamma_{2\ell-2}'$ connects the two horospheres $h_y \cap (m_1 \ldots m_{\ell-1}H_G)$ and $h_y \cap \varrho(\omega_0)H_G$ to each other. Similarly, $\gamma_{2\ell-2}'$ connects $h_y' \cap \varrho(\omega_0)H_G$ and $h_y' \cap (\varrho(\omega_0) p_{r_{\ell}}^{k_{\ell}} H_G)$ to each other. By choice of the sufficiently large powers of the $p_i$, these sets are at least $D$ away from each other so $\gamma_{2\ell-2}'$ and $\gamma_{2\ell}'$ have length at least $D$. 

It is important to be aware that if the last two parabolic subwords have the same fixed point, then $y=y'$ and $\gamma_{2\ell}'$ has length 0. However, because $\omega_1$ is in the normal form we described above, if $p_{r_{\ell-1}} = p_{r_{\ell}}$ the $m_\ell$ does not commute with $p_{r_{\ell}}$, and $p_{r_{\ell-1}}^{k_{\ell-1}} \neq m_\ell p_{r_{\ell}}^{k_{\ell}} (m_\ell)^{-1}$. Hence, $(m_1 \ldots m_{\ell-1}) p_{r_{\ell-1}}^{k_{\ell-1}} (m_1 \ldots m_{\ell-1})^{-1} \neq (\varrho(\omega_0)) p_{r_{\ell}}^{k_{\ell}} (\varrho(\omega_0))^{-1}$, and these parabolic subwords cannot translate around the same boundary point twice in a row. Thus, $\gamma'$ does not backtrack or self-intersect, and $x \neq \varrho(\omega(x))$. Then $\varrho(\omega)$ is not the identity and $\varrho$ is faithful. 
\end{proof}

Now that we have constructed our surface subgroup in $\Gamma$, we will show that this subgroup is Zariski dense by employing the following theorem:

    \begin{theorem}[Chen, Greenberg, 1974, \cite{CG}] \label{CG1974}
    Any subgroup of $\textnormal{SO}^+(n,1)$ that does not fix a point on the boundary of $\h^n$ and does not preserve a totally geodesic submanifold in $\h^n$ is either discrete or dense in $\textnormal{SO}^+(n,1)$.  
    \end{theorem}

We can use this theorem to prove a similar lemma about Zariski closures of subgroups of SO$^+(n,1)$.

\begin{lemma} \label{lemma:max.Z.closed} 

Let $\Lambda$ be a subgroup of $\textnormal{SO}^+(n,1)$ such that the Zariski closure $ZCL(\Lambda)$ contains a copy of $\textnormal{SO}^+(n-1,1)$. Suppose also that $ZCL(\Lambda)$ does not preserve any hyperplane in $\h^n$. Then $\Lambda$ is Zariski dense in $\textnormal{SO}^+(n,1)$. 

\end{lemma} 

\begin{proof}
    
    Let $Z=ZCL(\Lambda)$. By supposition $Z$ contains a copy of SO$^+(n-1,1)$ that preserves some hyperplane $\h^{n-1}$. Immediately this tells us that $Z$ is not discrete, it does not preserve any point in the boundary of $\h^n$, and it does not preserve any totally geodesic submanifold with codimension 2 or greater. Also, by supposition $Z$ does not preserve any hyperplane in $\h^n$. Thus, $Z$ does not preserve any totally geodesic submanifold. Thus, by Theorem~\ref{CG1974}, $Z$ is dense in SO$^+(n,1)$.

    Note that, being Zariski closed implies being closed, so $Z$ is closed and dense. Hence, $Z$ is equal to SO$^+(n,1)$.

\end{proof}

\begin{corollary}\label{main.cor} 
   $\Gamma$ has a thin surface subgroup.
\end{corollary} 

\begin{proof}

    We first show that $\varrho(\pi_1(DM))$ is Zariski dense, and then use this to show that $\varrho(\pi_1(D\Sigma))$ is Zariski dense as well. 
    
    Let $Z$ be the Zariski closure of $\varrho(\pi_1(DM))$. $Z$ contains the Zariski closure of $G$ in SO$^+(4,1)$, which is a copy of SO$^+(3,1)$ that preserves $H_G$. We will denote this group SO$(G)$. 
    
    Now $\varrho(\pi_1(DM))$ contains $p_0 G p_0^{-1}$, so $Z$ contains the Zariski closure of this group as well. Hence, $p_0 \textnormal{SO}(G) p_0^{-1}$ is a subgroup of $Z$. $\textnormal{SO}(G)$ can only preserve a single copy of $\h^3$ in $\h^4$, specifically $H_G$. However, $p_0 \textnormal{SO}(G) p_0^{-1}$ can only preserve the hyperplane $p_0(H_G)$. Hence, $Z$ does not preserve any hyperplane in $\h^4$. Thus, by Lemma~\ref{lemma:max.Z.closed}, $Z$ is Zariski dense in SO$^+(4,1)$.

    Now the fiber group $\pi_1(\Sigma)$ is a normal subgroup of $\pi_1(M)$, so it is Zariski dense in $\textnormal{SO}^+(3,1)$. Thus, the Zariski closure of $\varrho(\pi_1(D\Sigma))$ in SO$^+(4,1)$ contains the same copies of $\textnormal{SO}^+(3,1)$ as $Z$. Hence, it also satisfies the suppositions of Lemma~\ref{lemma:max.Z.closed}, so it is Zariski dense. 
\end{proof}

More generally any discrete faithful representation of the fundamental group of the double of any finite volume, cusped, orientable hyperbolic 3-manifold in $\textnormal{SO}^+(4,1)$ is Zariski dense.

\begin{corollary} \label{3mfld.lemma}
Let $M$ be a finite volume, cusped, orientable hyperbolic 3-manifold with double $DM$. Suppose that there is a discrete faithful representation $\rho : \pi_1(DM) \hookrightarrow \textnormal{SO}^+(4,1)$.  Then the image of $\rho$ is Zariski dense.
\end{corollary}

\begin{proof}

    $DM$ contains a hyperbolic $3$-manifold so any faithful representation of $\pi_1(DM)$ has some subgroup that is isomorphic to a lattice in some copy of SO$^+(3,1)$. Thus, it suffices to show that $ZCL(\rho(\pi_1(DM)))$ does not preserve any hyperplane in $\h^4$. 
    
    $DM$ cannot be hyperbolized because it is a compact 3-manifold with an essential 2-torus, so there is no discrete faithful representation of $\rho(\pi_1(DM))$ in any copy of $\textnormal{SO}^+(3,1)$. Hence, $\rho(\pi_1(DM))$ does not preserve any copy of $\h^3$ in $h^4$, so it is Zariski dense. 
\end{proof}

In fact, by Lemma~\ref{3mfld.lemma} because $M$ is virtually fibered, any discrete faithful representation of $\pi_1(DM)$ yields a surface group that is Zariski dense in $\textnormal{SO}^+(4,1)$.

\section{Extension to Higher Dimensions}
\label{ch:higher-dim}
Next, we generalize Corollary~\ref{main.cor} to higher dimensions. For the remainder of this section, let $X$ be a cusped, arithmetic, hyperbolic $n$-orbifold with fundamental group $\Gamma$. Notice that the entire proof of Theorem~\ref{main.theorem} can be done independent of the dimension of $X$, so the generalization follows this same strategy when $n>4$. In order to ensure thinness of the resulting surface subgroups in higher dimensions, we will slightly modify the representations we get as a result. We begin by building a chain of nested submanifolds in a finite cover of $X$ to start the construction.

\begin{lemma} \label{matryoshka} 
    Let $X$ be a cusped arithmetic hyperbolic $n$-orbifold, where $n\geq4$. Then $X$ has some finite cover $X''$ with fundamental group $\Gamma''$ such that $X''$ contains a chain of cusped, arithmetic, hyperbolic, immersed  submanifolds $M_3 \subset M_4 \subset \ldots \subset M_{n-1} \subset X''$, where $M_i$ has dimension $i$. Additionally, $\Gamma''$ contains a chain of nonuniform, arithmetic subgroups $G_3 < G_4 < G_5 < \ldots < G_{n-1} < \Gamma''$ such that $G_i$ is isomorphic to $\pi_1(M_i)$. 
    
    Furthermore, we may require that each manifold or orbifold in this chain has torus cusps, the ends of any given submanifold do not collide at infinity in $X''$, and that $M_3$ is fibered.
\end{lemma}

\begin{proof}
We prove this lemma via induction on $n$. The base case $n=4$, was handled in the proof of Theorem~\ref{main.theorem}. Now suppose that $n > 4$. By the Tits classification, $\Gamma$ has some conjugate that is commensurable with $\textnormal{SO}^+(q,\mathbb{Z})$ where $q$ is a diagonal integral quadratic form with signature $(n, 1)$. By simply removing one of the positive coefficients of $q$, we get a diagonal subform $f_{n-1}$ with signature $(n-1,1)$. Because $n\geq 5$, Meyer's theorem tells us that $f_{n-1}$ is isotropic \cite{Serre}. 
Note that $\textnormal{SO}^+(f_{n-1},\mathbb{Z}) < \textnormal{SO}^+(q,\mathbb{Z})$ via the upper-left corner representation. Hence, $\textnormal{SO}^+(f_{n-1},\mathbb{Z})\setminus\h^{n-1}$ can be immersed in $\textnormal{SO}^+(q,\mathbb{Z})\setminus \h^n$.

Because $X$ is abstractly commensurable to $\textnormal{SO}^+(q,\mathbb{Z})\setminus \h^n$, there is a finite cover $X'$ of $\textnormal{SO}^+(q,\mathbb{Z})\setminus \h^n$ that covers $X$ as well. The fundamental group of this cover is a finite index subgroup of $\textnormal{SO}^+(q,\mathbb{Z})$. Now $\textnormal{SO}^+(f_{n-1},\mathbb{Z})\setminus\h^{n-1}$ lifts to an immersed suborbifold in this cover, with fundamental group a finite index subgroup of $\textnormal{SO}^+(f_{n-1},\mathbb{Z})$. Additionally, we can pass to another finite cover of $X'$ to ensure that the immersed suborbifold lifts to an embedded submanifold where the ends do not collide at infinity (see \cite{MRS}). Finally, there is a third finite cover such that the cover and the lift of the submanifold both have torus cusps \cite{MRS}. Let $X''$ be this final finite cover of X. By the inductive hypothesis, the lift of $\textnormal{SO}^+(f_{n-1},\mathbb{Z})\setminus\h^{n-1}$ in $X''$ contains a chain of manifolds of the desired form. Hence, $X''$ contains a chain of submanifolds of the desired form. Similarly, by the inductive hypothesis, the finite index subgroup of $\textnormal{SO}^+(f_{n-1},\mathbb{Z})$ that makes up the fundamental group of this submanifold contains a chain of subgroups of the desired form. 
\end{proof}

This construction gives us a cusped surface group in $\Gamma$, specifically a representation of the fiber group of $M_3$. Now, as before we next need to carefully embed the folded double of $M_3$ into $X$ to build a Zariski dense surface group in $\Gamma$. 

\begin{figure}[!ht] 
\centering
\begin{tikzpicture}[]
\tikzstyle{every node}=[font=\small]

            \node[color=black] (a) at (.9,.44) {$FM_3$};

            \draw[rounded corners=25pt, color = black](5.2,-2.05)--(3.15,-.75)--(1.5,-1.5)--(0,0) -- (1.5,1.5)--(3.15,.75)--(4.2,1.4);

            \draw[rounded corners=10pt, color=black](5.4,-1.86)--(3.6,-.6)--(3.6,.6)--(4.3,1.2);

            \draw[color = red, rounded corners=8pt](5.32,-1.9)--(3.3,-.5)--(3.3,.5)--(4.25,1.25);
            \draw[color = red, rounded corners=10pt](5.22,-1.95)--(3.4,-.8)--(1.5,-.7)--(1.1,0) -- (1.5,.7)--(3.15,.5)--(4.2,1.3);


            \draw[rounded corners=30pt, color=blue](5.32,-2.2)--(3.55,-1.1)--(1.5,-2)--(-.5,0) -- (1.5,2)--(3.3,.8)--(4.25,1.475);

            \draw[rounded corners=10pt, color=blue](5.71,-1.9)--(4.1,-.6)--(3.9,.6)--(4.7,1.2);


            \node[color=blue] (a) at (2.8,1.5) {$M_4$};


            \draw[rounded corners=33pt, color = {rgb,255:red,60; green,180; blue,84}](5.32,-2.2)--(3.55,-1.4)--(1.5,-2.4)--(-.9,0) -- (1.5,2.4)--(3.55,1.4)--(5.5,2.7);

            \draw[rounded corners=10pt, color = {rgb,255:red,60; green,180; blue,84}](5.71,-1.9)--(4.3,-.4)--(4.3,.4)--(5.8,1.85);

            \draw[color = {rgb,255:red,60; green,180; blue,84}] (1.5,0.02) arc (175:315:.5cm and 0.25cm);
            \draw[color = {rgb,255:red,60; green,180; blue,84}] (2.25,-0.21) arc (-30:180:0.35cm and 0.15cm);


            \node[color = {rgb,255:red,60; green,180; blue,84}] (a) at (0,1.5) {$X$};


            \draw[color = blue](4.7,1.2)--(5.2,1.6);

            \draw[color = blue](4.25,1.475)--(5.1,2.1);

            \draw[gray, opacity = 0.5] (4.3,1.215) arc (0:60:0.2cm);
            \draw[blue] (4.725,1.23) arc (30:80:0.2cm);
            \draw[blue, opacity = 0.5] (4.59,1.32) arc (210:260:0.2cm);
            \draw[rounded corners=3pt] (4.29,1.2)--(4.59,1.45)--(4.59,1.25)--(4.2,1.4);
            \draw[color = red, rounded corners=3pt, opacity=.3](4.58,1.54)--(4.73,1.33)--(4.6,1.19)--(4.2,1.3);
            \draw[black] (4.49,1.35) ellipse (0.3 and 0.2);
            \draw[color = red] (4.225,1.23)--(4.58,1.54);

            \draw[color = {rgb,255:red,60; green,180; blue,84}](5.71,-1.9)--(6,-2.2);

            \draw[color = {rgb,255:red,60; green,180; blue,84}](5.32,-2.2)--(5.8,-2.5);

            \draw[gray, opacity = 0.5] (5.39,-1.85) arc (-5:-82:0.22cm);
            \draw[color = red, rounded corners=3 pt, opacity=.3](5.48,-2.25)--(5.63,-2.15)--(5.75,-2)--(5.32,-1.9);
            \draw[color = {rgb,255:red,60; green,180; blue,84}, opacity = 0.3] (5.6,-2.01) arc (-110:0:0.1cm);
            \draw[color = {rgb,255:red,60; green,180; blue,84}] (5.725,-1.93) arc (70:180:0.09cm);
            \draw[color = black, rounded corners=3pt] (5.4,-1.86)--(5.50,-2.15)--(5.65,-1.975)--(5.2,-2.05);
            \draw[color=black] (5.5,-2.05) ellipse (0.295 and 0.195);
            \draw[color = red](5.22,-1.95)--(5.48,-2.25); 

            \node (a) at (1.3,.05) {};
            \node (sigma) [color = red, below of=a, node distance=25pt] {$F\Sigma$ };

\end{tikzpicture}
\caption{Diagram of $F\Sigma \subset FM_3 \subset FM_4$ nested in $X''$. The pictured meridians of the torus cusp cross-sections of $X''$ and $M_4$ lift to horocycles in $\h_n$ that will be preserved by carefully chosen parabolic elements in $\Gamma''$}
\label{fig:matryoshka}
\end{figure}
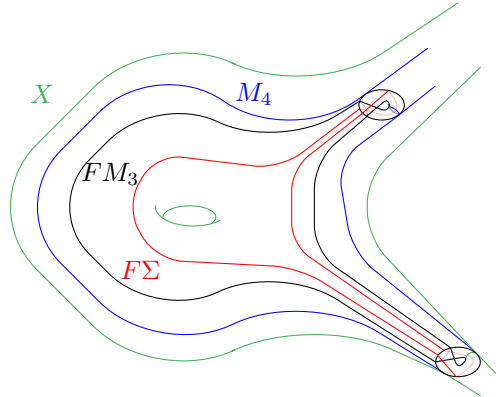

\begin{theorem}
\label{higher.dim}
Let $\Gamma$ be the fundamental group of a cusped, arithmetic, hyperbolic $n$-orbifold. $\Gamma$ has a thin surface subgroup.
\end{theorem}

\begin{proof}
We begin this proof by replacing $X$ with a suitable finite cover satisfying the conclusions of Lemma~\ref{matryoshka} so that $X$ contains a chain of nested submanifolds $M_3\subset\ldots\subset M_{n-1}\subset X$ with the claimed properties. Additionally, by passing to further finite sheeted covers, we may ensure that $M_3$ has arbitrarily many cusps, all of which do not collide at infinity in $M_k$ for $k\geq 4$ (see Proposition~3.1 in \cite{MRS}). We choose appropriate covers such that $M_3$ has $N$ cusps for $N > n-3$. Because the ends of $M_3$ do not collide at infinity in each of the $M_i$, this ensures that each $M_i$, hence $X$, has at least $N$ ends. We order the cusps of $M_3$ as: $C_0, C_1, \ldots C_{n-4}, \ldots C_{N-1}$.

Next, we will build a representation $\varrho:\pi_1(\fold_3)\hookrightarrow\Gamma$ by capping off the $2$-torus cusps with $3$-tori as before. We will do this by sending the stable letters to appropriately chosen parabolic elements in each maximal cusp subgroup of $\Gamma$, with the images of the stable letters $t_0,\ldots,t_{n-2}$ chosen carefully to ensure Zariski density of the image.

Let $P_{M_3,i}$ be a maximal cusp subgroup of $\pi_1(M_3)$ associated to the $i^{\textrm{th}}$ cusp. $M_3$ has 2-torus cusp cross-sections, so $P_{M_3,i} = \langle p_{(0,i)}, p_{(1,i)} \rangle$, where $p_{(0,i)}$ and $p_{(1,i)}$ are commuting parabolic elements of $\Gamma$. The cusp cross-sections of $M_3$ are essential submanifolds of the 3-torus cusp cross-sections of $M_4$. Hence, there is some parabolic element $p_{(2,i)}$ that commutes with $P_{M_3,i}$ such that $\langle p_{(0,i)}, p_{(1,i)}, p_{(2,i)} \rangle$ is a maximal cusp subgroup of $M_4$ associated to the cusp of $M_4$ that contains the $i^{th}$ cusp of $M_3$, which we call $P_{M_4,i}$. Inductively, because the $(k-1)$-torus cusp cross-sections of $M_k$ are essential submanifolds of the $k$-torus cross-sections of the cusps of $M_{k+1}$, we can denote $P_{M_{k+1},i} = \langle P_{M_k,i}, p_{(k-1,i)} \rangle$ by the maximal cusp subgroup of $M_{k+1}$ containing the cusp $P_{M_k,i}$. We can rewrite this as $P_{M_{k},i} = \langle p_{(0,i)}, p_{(1,i)}, \ldots p_{(k-2,i)} \rangle$. This gives us a chain of maximal cusp subgroups $  P_{M_3,i} < P_{M_4,i} < \ldots < P_{M_{n-1},i}$, for at least $N$ cusps of each of the submanifolds in our chain. In this vein we can let $ P_{X,i} = \langle p_{(0,i)}, p_{(1,i)}, \ldots p_{(n-2,i)} \rangle$. 

Now we can define our representation. Let $\varphi: \pi_1(M_3) \to G_3$ be the isomorphism we get from Lemma~\ref{matryoshka}. We can define $\varrho(\fold_3)$:
\begin{equation}
\varrho(x) = \begin{cases} 
    \varphi(x) & x \in \pi_1(M_3) \\
    p_{(i+2,i)} & x = t_i, \hspace{6pt} 0 \leq i \leq n-3\\
    p_{(2,i)} & x = t_i, \hspace{6pt} i \geq n-2\\
    
   \end{cases}
\end{equation}
Now, we note that the argument in the proof of Theorem~\ref{main.theorem} does not rely on the dimension of $X$, so long as it is at least $4$. Thus, we may now repeat the proof of Theorem~\ref{main.theorem} verbatim and conclude that, after passing to large enough powers of the $p_{(j,i)}$, $\varrho(\pi_1(\fold_3))$ is faithful. All that remains is to verify that the image of $\pi_1(D\Sigma)$ is Zariski-dense in $\textnormal{SO}(q,\mathbb{R}) \cong \textnormal{SO}^+(n,1)$, where $\Sigma$ is the fiber surface in $M_3$. 

We first show that $\pi_1(DM_3)$ is Zariski dense in $\textnormal{SO}(q,\mathbb{R})$ by considering its action on $\h^n$. Broadly, this will follow from showing that the set of parabolic elements $\{p_{(i+2,i)}\,|\,0\leq i\leq n-4\}$ translates a copy of $\h^3$ ``in every direction'' in $\h^n$. To start, we identify particular subspaces of $\h^n$ fixed by the submanifolds of $X$ coming from Lemma~\ref{matryoshka}. Starting from the manifold of least codimension, $G_{n-1}$ fixes exactly one copy of $\h^{n-1}$ inside $\h^n$, which we denote by $H_{n-1}$. We denote by $\textnormal{SO}^+(G_{n-1})\cong\textnormal{SO}^+(n-1,1)$ the stabilizer of $H_{n-1}$ in $\textnormal{SO}^+(n,1)$. Moving down one dimension, note that $G_{n-2}$ can preserve multiple copies of $\h^{n-2}$ inside $\h^n$, but as it is itself a subgroup of $G_{n-1}$, it preserves exactly one copy of $\h^{n-2}$ inside $H_{n-1}$. We similarly denote this copy by $H_{n-2}$ and its stabilizer inside $\textnormal{SO}^+(G_{n-1})$ by $\operatorname{SO}^+(G_{n-2})$. Continuing in this manner, we get a chain of hypersurfaces $H_3\subset H_4\subset\ldots\subset H_{n-1}\subset\h^n$ such that each $H_i$ has stabilizer $\textnormal{SO}^+(G_i)$ and is preserved by $G_i$.

Now, we proceed inductively and show that the group
\[
S_k = \langle G_3,p_{(2,0)}G_3p_{(2,0)}^{-1},\ldots,p_{(k-2,k-4)}G_3p_{(k-2,k-4)}^{-1}\rangle
\]
for $4\leq k\leq n-1$ is Zariski dense in $\textnormal{SO}^+(G_k)$. For the base case, we set $S_3:= G_3$ and note that, as this is the fundamental group of $M_3$, this group is Zariski dense in $\textnormal{SO}^+(G_3)$. We also illustrate the case for $S_4$ for concreteness. The parabolic element $p_{(2,0)}\in G_4$ is chosen so that it does not preserve $H_3$, hence, just as in the proof of Lemma~\ref{main.cor}, we have that $S_4=\langle G_3,p_{(2,0)}G_3 p_{(2,0)}^{-1}\rangle$ is Zariski dense in $\textnormal{SO}^+(G_4)$.

Now, we proceed with the inductive step. Suppose that $S_k$ is Zariski dense in $\textnormal{SO}^+(G_k)$. By construction, $G_k$ preserves $H_k$ but no other copy of $\h^k$ inside $H_{k+1}$. Therefore, by the inductive hypothesis, $S_k$ also preserves $H_k$ but no other copy of $\h^k$ inside $H_{k+1}$. However, the element $p_{(k-1,k-3)}$ does not preserve $H_k$ by choice, so $S_{k+1}$ cannot preserve $H_k$. Therefore, $S_{k+1}$ does not preserve any copy of $\h^k$ inside $H_{k+1}$. Since, by the inductive hypothesis, the Zariski closure of $S_{k+1}$ also contains a copy of $\textnormal{SO}^+(k,1)$, by Lemma~\ref{lemma:max.Z.closed}, it is Zariski dense in $\textnormal{SO}^+(G_{k+1})$. Therefore, $S_n<\varrho(\pi_1(DM_3))$ is Zariski dense in $\textnormal{SO}^+(n,1)$.

Finally, as in the proof of Corollary~\ref{main.cor}, the fiber group $\pi_1(\Sigma)$ is Zariski dense in $\textnormal{SO}^+(3,1)$ and, by a similar argument as above, $\varrho(\pi_1(D\Sigma))$ does not preserve any hyperplane in $\h^n$ and contains a copy of $\textnormal{SO}^+(n-1,1)$ and hence, again by Lemma~\ref{lemma:max.Z.closed}, is Zariski dense in $\textnormal{SO}^+(n,1)$.
\end{proof}

\section{GFERF Subgroups}
\label{ch:GFERF}
We say that a group of isometries of $\h^n$ is \emph{GFERF}, or \emph{geometrically finite locally extended residually finite} if it is subgroup separable on all its geometrically finite subgroups. This is a slightly weaker condition than being separable on all finitely generated subgroups, or being LERF\@. In this section, we prove Theorem~\ref{12}, restated here.

\begin{theorem}
Let $M$ be a cusped, arithmetic, hyperbolic $n$-orbifold and let $DM$ be the double of $M$ over its cusps. Then $\pi_1(DM)$ embeds discretely into $\operatorname{SO}^+(n+1,1)$ and its image is GFERF.
\end{theorem}

\begin{proof} 
Subgroup separability is a property that can be passed easily to subgroups. That is to say that if we have groups $G,H,K$ such that $H < K < G$, if $G$ is $H$-subgroup separable then $K$ is $H$-subgroup separable. This fact follows from the definition of subgroup separability [See \cite{ALR} for details]. Additionally, if $G,H$, and $K$ act by isometries on $\h^n$, the choice of supergroup that contains $H$ has no effect on its fundamental domain in $\h^n$. Hence, if $H$ is geometrically finite as a subgroup of $K$ then it is geometrically finite as a subgroup of $G$. Therefore, subgroups of GFERF groups are themselves GFERF.

The fundamental groups of cusped, arithmetic, hyperbolic, orbifolds are GFERF \cite{BHW}. Thus, if there is a faithful representation of $\pi_1(DM)$ in $\pi_1(X)$, where $X$ is a cusped arithmetic hyperbolic $(n+1)$-orbifold then $\pi_1(DM)$ is GFERF\@. Therefore, the fundamental groups of the doubles of the fibered 3-manifolds that we considered in Section~\ref{ch:folded.double} are GFERF\@. We now extend this result to the doubles of all cusped arithmetic hyperbolic 3-manifolds and orbifolds, as well as higher dimensional manifolds and orbifolds.

We proceed by using the folded double, as in Section~\ref{ch:folded.double}. Though we initially defined $FM$ for finite-volume, cusped, hyperbolic 3-manifolds, the construction pictured in Figure~\ref{fig:graph.fold} and the proof pictured in Figure~\ref{fig:graph.isomorphisms} are both independent of the dimension of $M$. So any faithful representation of $\pi_1(FM)$ in some $\pi_1(X)$ restricts to a faithful representation of $\pi_1(DM)$ for $M$ any finite-volume, cusped, hyperbolic $n$-manifold. 

In order to construct this representation, we first find some $X$ such that a finite sheeted cover $M'$ of $M$ has a $\pi_1$-injective immersion or embedding into $X$. When $n > 3$ we do this by noting that some conjugate of $\pi_1(M)$ is commensurable to $\textnormal{SO}^+(f, \mathbb{Z})$ for some form $f$ of signature $(n,1)$. Then we can let $M' = \textnormal{SO}^+(f, \mathbb{Z}) \setminus \h^n$ and $X = \textnormal{SO}^+(q, \mathbb{Z}) \setminus \h^n$, where $q = \langle f \rangle \oplus \langle 1 \rangle$, which allows us to use the upper left corner function as in the proof of Lemma~\ref{lemma:groups.and.qf} to give us a $\pi_1$-injective immersion of $M'$ into $X$. 
When $n = 3$, we can realize $\pi_1(M)$ as a Bianchi group, thus we can pass to a finite sheeted cover of $M$ that is a manifold that embeds as a totally geodesic submanifold of a cusped orientable arithmetic hyperbolic 4-manifold [See \cite{KRS} corollary 1.5]. This embedding is $\pi_1$-injective.

Now we can pass to finite covers of $M'$ and $X$ such that $M'$ is embedded in $X$, $M'$ has $(n-1)$-torus cusps, $X$ has $n$-torus cusps, and the cusps of $M'$ do not collide at infinity. This allows us to recreate the function $\varrho:\pi_1(FM') \hookrightarrow \pi_1(X) $ from the proof of Theorem~$\ref{higher.dim}$. In fact, the proof of Theorem~$\ref{higher.dim}$ needs only slight changes to generalize to higher dimensions. If we replace $\h^4$ with $\h^{n+1}$, $\h^3$ with $\h^n$, 3-manifolds with $n$-manifolds, and 2-manifolds with $(n-1)$-manifolds then this proof still holds. Thus, there is a faithful representation of $\pi_1(FM')$ in $\pi_1(X)$, and $M$ has some finite cover $M'$ such that $\varrho(\pi_1(DM'))$ is GFERF. 

Now, we would like to show that $\pi_1(DM)$ itself has GFERF image. Note that the covering map from $M'$ to $M$ induces a finite degree covering map from $DM'$ to $DM$. This cover is commensurable with the base space. Commensurability preserves GFERF, so $\pi_1(DM)$ has GFERF image. 
\end{proof}

A more intrinsic statement of this result, independent of the embedding of $\pi_1(DM)$ can also be given. First, let $\mathcal{P}$ denote the collection of cusp subgroups of $\pi_1(DM)$. Then $\pi_1(DM)$ is hyperbolic relative to the family $\mathcal{P}$. A subgroup $H\leqslant \pi_1(DM)$ is geometrically finite (under the above embedding into $\operatorname{SO}^+(n+1,1)$) if and only if $H$ is relatively quasiconvex in $\pi_1(DM)$, relative to $\mathcal{P}$ (c.f. \cite{hruska}). Recall that a (relatively) hyperbolic group is said to be (relatively) QCERF, or \textit{(relatively) quasi-convex extended residually finite}, if it is subgroup separable on its (relatively) quasiconvex subgroups. Therefore, by the above methods, we also arrive at the following.

\begin{corollary}
    $\pi_1(DM)$, as above, is QCERF relative to $\mathcal{P}$.
\end{corollary}

\bibliography{references.bib}

\appendix
\section{Computations of Hasse-Minkowski Invariants}
\label{append:hasse}

In this appendix we compute the Hasse-Minkowski invariants of all diagonal quadratic forms with signature $(4,1)$ and coefficients in $\mathbb{Z}$. The definition of Hasse-Minkowski invariants relies on the Hilbert Symbol: 

\begin{definition}
    The \emph{Hilbert Symbol} over $\mathbb{Q}_p$ of two non-zero elements $a$ and $b$ is a function $(a,b)_p$ that characterizes the quadratic form $q = \langle a,b,-1 \rangle$. It is defined:

\[
(a,b)_p=
\begin{cases}
    1 & q \textrm{\ is isotropic}\\
    -1 & q  \textrm{\ is anisotropic}
\end{cases}
\]

\end{definition}

\begin{definition}
    Let $q = \langle a_0, a_1, a_2, a_3, a_4 \rangle$ be a non-degenerate, diagonal  quadratic over the local field $\mathbb{Q}_p$. The \emph{Hasse-Minkowski invariant} $c_p(q)$ is the product of the Hilbert symbols of the coefficients. It can be computed with the formula \[c_p(q)= \prod_{i<j} (a_i,a_j)_p\]
\end{definition}

Note that as the name suggests Hasse-Minkowski invariants are invariant within projective equivalence classes of quadratic forms, so we only need to compute these values for representative elements of each equivalence class. Hence, we can restrict to the family of forms from \cite{Montesinos} used in Section~\ref{groups.and.qf}.

For the computations in this section, we find it convenient to use the following formula for Hilbert symbols \cite{Serre}. Let $a = p^\alpha*u$, $b = p^\beta*v$, and let $\epsilon(u) \equiv  \frac{1}{2}(u-1) \mod 2$ and $\omega(u) \equiv \frac{1}{8}(u^2-1) \mod 2$. Then

\begin{equation} \label{hilbert}
(a,b)_p = \begin{cases} 
    (-1)^{\epsilon(u)\epsilon(v) +\alpha \omega(v) + \beta \omega(u)} & p=2 \\
    (-1)^{\alpha \beta \epsilon(p)}\cdot \legendre{u}^\beta \cdot \legendre{v}^\alpha & p \neq 2\\
   \end{cases}
\end{equation}
where $\legendre[\cdot]{\cdot}$ is the Legendre symbol. We will also use the following properties \cite{Serre}: 
\begin{enumerate}
    \item $(a,1)_p = 1$
    \item $(a,-a)_p = 1$
    \item $(a,b)_p = (b,a)_p$
    \item $(ab,c)_p = (a,c)_p(b,c)_p$.
\end{enumerate}

By definition, if $S \equiv 1 \mod 4$ then $c_p(q) = (-1,aS)_p(aS,a)_p(-1,a)_p $. From the properties listed above, we get $c_p(q) = (S,-a)_p(-1,a)_p$. Similarly, when $S \equiv -1 \mod 4$ we know that $c_p(q) = (aS,-a)_p = (S,-a)_p$. Immediately this gives us the following Hasse-Minkowski invariants:

\[c_p(q)= \begin{cases} 
      (-1)^{\tfrac{-a-1}{2}\cdot\tfrac{S-1}{2}} & p = 2 \\
      (S,-a)_p & p = a, \hspace{4pt} a \equiv 1 \mod 4 \\
      (-1)\cdot(S,-a)_p & p = a, \hspace{4pt} a \equiv -1 \mod 4 \\
      (S,-a)_p & p | S \\
      1 & else \\
   \end{cases}
\] 

Then the only value left to compute is $(S,-a)_p$. Applying Equation~\ref{hilbert} we see that when $p$ divides $S$, this is \legendre{-a}, which is $-1$ by assumption. When $p=a$, we have $(S,-a)_a = \legendre[a]{S}$. We wish to turn this into a product of $\legendre[p_i]{-a}$ for $p_i$ dividing $S$, so we will use the following properties of the Legendre symbol \cite{Gauss}:

\begin{enumerate}
    \item Definition: $\legendre{-1} = (-1)^{\tfrac{p-1}{2}}$
    \item Multiplicativity: $\legendre{a}\legendre{b} = \legendre{ab}$
    \item Quadratic reciprocity: $\legendre[a]{b}\legendre[b_{ }]{a}=(-1)^{{\tfrac {a-1}{2}}\cdot {\tfrac {b-1}{2}}}$ for $a$ and $b$ distinct odd primes
\end{enumerate}

Beginning with quadratic reciprocity, we see that:

\[
\begin{array}{*{3}{rcl}l}
    \legendre[a_{ }]{p_i}\legendre[p_i]{a}\legendre[p_i]{-1} & = &  (-1)^{{\tfrac {a-1}{2}}\cdot {\tfrac {b-1}{2}}}\legendre[p_i]{-1}\\
    \legendre[a_{ }]{p_i}\legendre[p_i]{-a} & = &  (-1)^{{\tfrac {a-1}{2}}\cdot {\tfrac {p_i-1}{2}}}(-1)^{\frac{p_i-1}{2}}\\
    \legendre[a_{ }]{p_i} & = &  (-1)^{{\tfrac {a-1}{2}}\cdot {\tfrac {p_i-1}{2}}+\tfrac{p_i-1}{2}+1}\\
    \legendre[a_{ }]{S} & = &  \prod_{p_i | S}(-1)^{{\tfrac {a-1}{2}}\cdot {\tfrac {p_i-1}{2}}+\tfrac{p_i-1}{2}+1}\\
    \legendre[a_{ }]{S} & = &  (-1)^{\sum_{p_i | S}({\tfrac {a-1}{2}}\cdot {\tfrac {p_i-1}{2}}+\tfrac{p_i-1}{2}+1)}\\
\end{array}
\]

 We want to compute the parity of the sum above. Thus, our computation relies on $a \mod 4$ and $p_i \mod 4$. We can use our formulation of $S$ to determine these values. Note that when $p_i \equiv 1 \mod 4$ the term ${{\tfrac {a-1}{2}}\cdot {\tfrac {p_i-1}{2}}+\tfrac{p_i-1}{2}+1}$ is odd. Similarly, when $p_i \equiv -1 \equiv a \mod 4$ this term is odd. However, when $p_i \equiv -1 \equiv -a \mod 4$ it is even.  Now, let $S_0$ be the set of $p_i$ such that $p_i \equiv 1 \mod 4$ and $S_1$ be the set of $p_i$ such that $p_i \equiv -1 \mod 4$. The parity of the sum is exactly the parity of $\abs{S_0} + \abs{S_1}\cdot\frac{a-1}{2} $. 
 
 If $S \equiv 1 \mod 4$, then $\abs{S_1}$ is even. So the parity of the sum in question is exactly the parity of $\abs{S_0}$. Furthermore, by our initial assumptions,  $a \equiv (-1)^{\abs{S_0}+\abs{S_1}} \equiv (-1)^{\abs{S_0}} \mod4$, so it is also the parity of $\frac{a-1}{2}$. Hence, we have that $\legendre[a_{ }]{S} = (-1)^{\frac{a-1}{2}}$. 

Conversely, if $S \equiv -1\mod 4$, then $\abs{S_1}$ is odd. In this case by our assumptions on $a$ we have that $a \equiv (-1)^{\abs{S_0}+\abs{S_1}+1} \equiv (-1)^{\abs{S_0}} \mod4$. Hence, ${\frac{a-1}{2}}\cdot \abs{S_1}$ is even if and only if $\abs{S_0}$ is even. Then the sum of these two terms is always even, and $\legendre[a_{ }]{S} = 1$. This gives us the Hasse-Minkowski invariants in terms of $S \mod 4$ and $a\mod 4$:

\[c_p(q)= \begin{cases} 
      1 & p = 2, \hspace{4pt} S \equiv 1 \mod 4 \\
      -1 & p = 2, \hspace{4pt} S \equiv -1 \mod 4, \hspace{4pt} a \equiv 1 \mod 4 \\
      1 & p = 2, \hspace{4pt} S \equiv -1 \mod 4, \hspace{4pt} a \equiv -1 \mod 4 \\
      1 & p = a, \hspace{4pt} S \equiv 1 \mod 4, \hspace{4pt} a \equiv 1 \mod 4 \\
      -1 & p = a, \hspace{4pt} S \equiv 1 \mod 4, \hspace{4pt} a \equiv -1 \mod 4 \\
      1 & p = a, \hspace{4pt} S \equiv -1 \mod 4 \\
      -1 & p | S \\
      1 & else \\
   \end{cases}
\]

\end{document}